\documentclass[10pt]{amsart}
\usepackage[latin1]{inputenc}
\usepackage{amsmath,amssymb,amsthm,
	verbatim,
	fancyhdr,amsfonts}

\usepackage[format=plain,labelfont=it,textfont=it]{caption}
\usepackage{needspace}
\usepackage{float}
\usepackage{tikz}
\usetikzlibrary{decorations.pathreplacing}

\usepackage[pdftex]{lscape}
\usepackage{amsmath}
\usepackage{amsthm}
\usepackage{amsfonts,amsmath,amssymb}
\usepackage{enumerate,verbatim,color}
\usepackage{epsfig}
\usepackage{textcomp}
\usepackage{graphicx}
\usepackage{amsfonts}
\usepackage{mathrsfs}
\usepackage{upgreek}
\usepackage{mathtools}
\usepackage{pdfpages}

\newcommand{\virgolette}[1]{``#1''}
\newtheorem{teorema}{Theorem}[section]
\newtheorem{coro}[teorema]{Corollary}
\newtheorem{lemma}[teorema]{Lemma}
\newtheorem{prop}[teorema]{Proposition}
\newtheorem{osss}[teorema]{Remark}

\newtheorem*{oss}{Remark}
\theoremstyle{definition}
\newtheorem{defi}[teorema]{Definition}
\theoremstyle{remark}
\newtheorem*{assumptionI}{\bf Assumptions I}
\newtheorem*{assumptionII}{\bf Assumptions II}

\newcommand{\iii}{{\, \vert\kern-0.25ex\vert\kern-0.25ex\vert\, }}

\newcommand{\ffi}{\varphi}

\newcommand{\AL}{\mathcal{A}}

\newcommand{\LL}{\mathcal{L}}
\newcommand{\MM}{\mathcal{M}}

\newcommand{\NN}{\mathcal{N}}
\newcommand{\KK}{\mathcal{K}}

\newcommand{\Di}{\mathcal{D}}

\newcommand{\N}{\mathbb{N}}
\newcommand{\R}{\mathbb{R}}

\newcommand{\C}{\mathbb{C}}

\newcommand{\Hi}{\mathscr{H}}

\newcommand{\Gi}{\mathscr{G}}

\newcommand{\dd}{\partial}

\newcommand{\ra}{\rangle}
\newcommand{\la}{\langle}

\newcommand{\G}{\Gamma}
\newcommand{\Ti}{\mathcal{T}}

\begin{document}

\date{}
\title[Bilinear control]{Controllability of localized quantum states on infinite graphs through bilinear control fields}

\author{Ka\"{\i}s Ammari}
\address{UR Analysis and Control of PDEs, UR 13ES64, Department of Mathematics, Faculty of Sciences of Monastir, University of Monastir, Tunisia}
\email{kais.ammari@fsm.rnu.tn} 

\author{Alessandro Duca}
\address{Institut Fourier, Université Grenoble Alpes, 100 Rue des Mathématiques, 38610 Gières, France} 
\email{alessandro.duca@unito.it}

\begin{abstract}
In this work, we consider the bilinear Schr\"odinger equation (\ref{mainx1}) $i\dd_t\psi=-\Delta\psi+u(t)B\psi$ in the Hilbert space $L^2(\Gi,\C)$ with $\Gi$ an infinite graph. The Laplacian $-\Delta$ is equipped with self-adjoint boundary conditions, $B$ is a bounded symmetric operator and $u\in L^2((0,T),\R)$ with $T>0$. We study the well-posedness of the (\ref{mainx1}) in suitable subspaces of $D(|\Delta|^{3/2})$ preserved by the dynamics despite the dispersive behaviour of the equation. In such spaces, we study the global exact controllability and the {\virgolette{energetic controllability}}. We provide examples involving for instance infinite tadpole graphs.  
\end{abstract}

\subjclass[2010]{35Q40, 93B05, 93C05}
\keywords{Bilinear control, infinite graph}

\maketitle

{\centering \small{Online published https://doi.org/10.1080/00207179.2019.1680868}\\
\ \ \ \ \ \ \ \ In this version, the references of the preprint 
articles cited in the original work were updated.}
\tableofcontents
\newpage
\section{Introduction}
We study the evolution of a particle confined in an infinite graph structure and subjected to an external field that plays the role of a control. 

\begin{figure}[H]
	\centering
	\includegraphics[width=\textwidth-50pt]{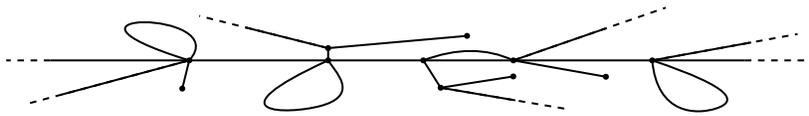}
	\caption{An infinite graph is an one-dimensional domain composed by vertices (points) connected by edges (segments and half-lines).}\label{fig:1}
\end{figure}

\noindent
Its dynamics is described by the so-called bilinear Schr\"odinger equation
\begin{equation}\label{bilinear} i\dd_t\psi(t)=(A+u(t)B)\psi(t),\ \ \ \ \ t\in(0,T),\end{equation} 
in $L^2(\Gi,\C)$, where $\Gi$ is the graph. The operator $A$ is a self-adjoint Laplacian, while the action of the controlling external field is given by the bounded symmetric operator $B$ and by the function $u$, which accounts its intensity. We call $\G_t^u$ the unitary propagator generated by $A+u(t)B$ (when it is defined).

%
%
%
%

\smallskip
It is natural to wonder whether, given any couple of states $\psi^1$ and $\psi^2$, there exists $u$ steering the bilinear quantum system from $\psi^1$ into $\psi^2$. The bilinear Schr\"odinger equation is said to be {\it exactly controllable} when the dynamics reach precisely the target. 

\noindent
We denote it {\it approximately controllable} when it is possible to approach the target as close as desired. If it is possible to control (either exactly, or approximately) more initial states at the same time with the same $u$, then the equation is said to be {\it simultaneously controllable}.

\smallskip

The controllability of finite-dimensional quantum systems ({\it i.e.} modeled by an ordinary differential equation) is currently well-established. If we consider the bilinear Schr\"odinger equation $(\ref{bilinear})$ in $\C^N$ such that $A$ and $B$ are $N\times N$ Hermitian matrices and $t\mapsto u(t)\in\R$ is the control, then the controllability of the problem is linked to the rank of the Lie algebra spanned by $A$ and $B$ (we refer to \cite{basso} by Altafini and 
\cite{corona} by Coron).
Nevertheless, the Lie algebra rank condition can not be used for infinite-dimensional quantum systems (see \cite{corona} for further details). Thus, different techniques were developed in order to deal with this type of problems.

\smallskip

Regarding the linear Schr\"odinger equation, the controllability and observability properties are reciprocally dual (often referred to the Hilbert Uniqueness Method). One can therefore address the control problem directly or by duality with various techniques: multiplier methods (\cite{lion}), 
microlocal analysis (\cite{rauch}), 
Carleman estimates 
(\cite{330}).

\smallskip

Even though the linear Schr\"odinger equation is widely studied in the literature, the bilinear Schr\"odinger equation in a generic Hilbert space $\Hi$ can not be approached with the same techniques since it is not exactly controllable in $\Hi$. We refer to the work on bilinear systems \cite{ball} by Ball, Mardsen and Slemrod, where the well-posedness and the non-controllability are provided. 
Despite they prove the well-posedness of the bilinear Schr\"odinger equation in $\Hi$ when $u\in L^1((0,T),\R)$ and $T>0$, they also show that it is not exactly controllable in $\Hi$ for $u\in L^2_{loc}((0,\infty),\R)$ (see \cite[Theorem\ 3.6]{ball}).

\smallskip

Because of the Ball, Mardsen and Slemrod result, many authors have considered weaker notions of controllability when $\Gi=(0,1)$. Let
$$D(A_D)=H^2((0,1),\C)\cap H^1_0((0,1),\C)),\ \ \ \ \ A_D\psi:=-\Delta\psi,\ \ \ \ \forall\psi\in D(A_D).$$

\noindent
In \cite{laurent}, Beauchard and Laurent prove the {\it well-posedness} and the {\it local exact controllability} of the bilinear Schr\"odinger equation in $H^{s}_{(0)}:=D(A_D^{s/2})$ for $s= 3$, when $B$ is a multiplication operator for suitable $\mu\in H^3((0,1),\R)$. 

\noindent
In \cite{morgane1}, Morancey proves the {\it simultaneous local exact controllability} of two or three (\ref{bilinear}) in $H^3_{(0)}$ for suitable operators $B=\mu\in H^3((0,1),\R)$.

\noindent
In \cite{morganerse2}, Morancey and Nersesyan extend the previous result. They achieve the {\it simultaneous global exact controllability} of finitely many (\ref{bilinear}) in $H^4_{(0)}$ for a wide class of multiplication operators $B=\mu$ with $\mu\in H^4((0,1),\R)$.

\noindent
In \cite{mio1}, the author ensures the {\it simultaneous global exact controllability in projection} of infinite (\ref{bilinear}) in $H^3_{(0)}$ for bounded symmetric operators $B$.

\noindent
The author exhibits the {\it global exact controllability} of the bilinear Schr\"odinger equation between eigenstates via explicit controls and explicit times in \cite{mio2}.

\smallskip

The {\it global approximate controllability} of the bilinear Schr\"odinger equation is proved with many different techniques in literature as the following.
The outcome is achieved with Lyapunov techniques by Mirrahimi in \cite{milo} and by Nersesyan in \cite{nerse2}.
Adiabatic arguments are considered by Boscain, Chittaro, Gauthier, Mason, Rossi and Sigalotti in \cite{ugo2} and \cite{ugo3}.
Lie-Galerking methods are used by Boscain, Boussa\"id, Caponigro, Chambrion and Sigalotti in \cite{nabile} and \cite{ugo}.

\smallskip

Control problems involving networks have been very popular in the last decades, however the bilinear Schr\"odinger equation on compact graphs has been only studied in \cite{mio3} and \cite{mio4}. In the mentioned works, the {\it well-posedness} and the {\it global exact controllability} of the (\ref{bilinear}) are provided in some spaces $D(|A|^{s/2})$ with $s\geq 3$. In \cite{mio4}, another weaker result is introduced, the so-called {\it energetic controllability}. In particular, a bilinear quantum system is said to be energetically controllable with respect to some energy levels when there exist corresponding bounded states $\{\ffi\}_{j\in\N^*}$ such that$$\forall m,n\in\N^*,\ \exists T>0,\ u\in L^2((0,T),\R)\ \ :\ \ \ffi_n=\G_T^u\ffi_m.$$
The peculiarity of the bilinear Schr\"odinger equation on compact graphs is that, even though $A$ admits purely discrete spectrum $\{\lambda_k\}_{k\in\N^*}$ (see \cite[Theorem\ 18]{kuk}), the uniform gap condition $\inf_{{k\in\N^*}}|\lambda_{k+1}-\lambda_k|\geq 0$ is satisfied if and only if $\Gi=(0,1)$. This hypothesis is crucial for the classical arguments adopted in the previous works as \cite{laurent}, \cite{mio1}, \cite{mio2} and \cite{morgane1}. To this purpose, new techniques are developed in \cite{mio3} and \cite{mio4} in order to achieve controllability results.

\subsection{Novelties of the work}

Up to our knowledge, the controllability of the bilinear \\
Schr\"odinger equation on infinite graphs is still an open problem.
The main reason can be found on the dispersive phenomena characterizing the equation on infinite graphs (not considering the difficulties already appearing on compact graphs; see \cite{mio3} and \cite{mio4}). 
A characteristic feature of the Schr\"odinger equation is the loss of localization of the wave packets during the evolution, the dispersion. This effect can be measured by $L^\infty$-time decay, which implies a spreading out of the solutions, due to the time invariance of the $L^2$-norm. In \cite{AAN17}, Ali Mehmeti-Ammari-Nicaise prove that the free Schr\"odinger group on the tadpole graph satisfies the standard $L^1 - L^\infty$ dispersive estimate and that it is independent of the length of the circle (compact part of the graph) (see also \cite[Ali Mehmeti-Ammari-Nicaise]{AAN15} for the case of the star-shaped network and with potential). The proof of this result is based on an appropriate decomposition of the kernel of the resolvent. This technique gives a full characterization of the spectrum made of the point spectrum and of the absolutely continuous one, while the singular continuous spectrum is empty. 

\medskip

Our strategy can be resumed as follows.

\begin{itemize} \item When $A$ has discrete spectrum, we construct some eigenfunctions of $A$ in $L^2(\Gi,\C)$ denoted $\{\ffi_k\}_{k\in\N^*}$. The flow of the Schr\"odinger equation $i\dd_t\psi=A\psi$ preserves $$\widetilde \Hi=\overline{span\{\ffi_k:\ k\in\N^*\}}^{ L^2}.$$ 
	\item When $B$ stabilizes the space $\widetilde\Hi,$ the bilinear Schr\"odinger equation is {\it well-posed} in $\widetilde \Hi$ and in $D(|A|^\frac{s}{2})\cap\widetilde\Hi$ for suitable $s>0$ when $B$ is sufficiently regular. 
	\item In such space, we study the {\it global exact controllability} and the {\it energetic controllability} with respect to $\{\ffi_k\}_{k\in\N^*}$ by adapting the techniques developed for the compact graphs in \cite{mio3} and \cite{mio4}.
\end{itemize}

In the first part of the work, we consider a specific $B$ localized on the \virgolette{head} of an infinite tadpole $\Gi$. The chosen $B$ is symmetric with respect to the natural symmetry axis $r$ of $\Gi$ and we denote $\widetilde \Hi$ the space of those $L^2(\Gi,\C)$-functions that are antisymmetric with respect to $r$ (see Figure \ref{tadpole}). We prove the {global exact controllability} in $D(|A|^\frac{3}{2})\cap\widetilde\Hi$.

In the second part, we generalize the results for generic graphs and we apply them for those $\Gi$ containing a star graph (Section $\ref{exampleNOTtadpole}$). 
	
	\begin{figure}[H]
		\centering
		\includegraphics[width=\textwidth-50pt]{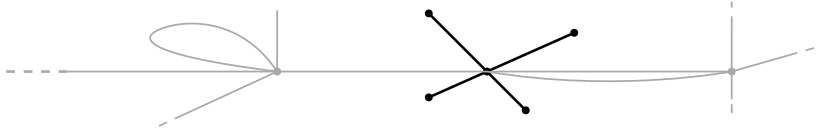}
		\caption{Graph described in Section $\ref{exampleNOTtadpole}$.}
	\end{figure}
In presence of suitable substructures in an infinite graph $\Gi$, it is possible to construct eigenfunctions of $A$. For instance, when $\Gi$ contains a self-closing edge $e$ of length $1$, the functions  $$\{\ffi_k\}_{k\in\N^*}\ \ \ \ \ :\ \ \ \ \ffi_k\big|_{e}=\sqrt{{2}}\sin\big({2k\pi x }\big),\ \ \ \ \ \ \ \ \ \ffi_k\big|_{\Gi\setminus \{e\}}\equiv 0,\ \ \ \ \ \ \ \forall k\in\N^*,$$ are eigenfunctions of $A$. If $B$ preserves the span of $\{\ffi_k\}_{k\in\N^*}$, 
then the controllability could be achieved. 
The same argument is true for graphs containing more self-closing edges or other suitable substructures (see Remark $\ref{generic}$ for few examples).

\section{Infinite tadpole graph}
Let $\Ti$ be an {\it infinite tadpole graph} composed by two edges $e_1$ and $e_2$. The self-closing edge $e_1$, the \virgolette{head}, is connected to $e_2$ in the vertex $v$ and it is parametrized in the clockwise direction with a coordinate going from $0$ to $1$ (the length of $e_1$). The \virgolette{tail} $e_2$ is a half-line equipped with a coordinate starting from $0$ in $v$ and going to $+\infty$.
\begin{figure}[H]
	\centering
	\includegraphics[width=\textwidth-100pt]{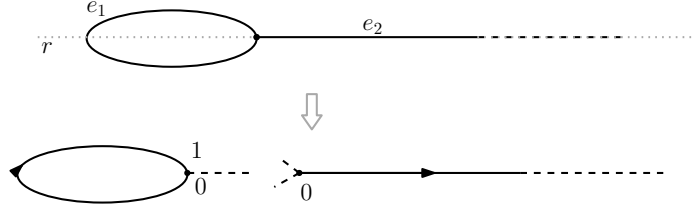}
	\caption{The parametrization of the infinite tadpole graph and its symmetry axis $r$.}
    \label{tadpole}
\end{figure}

\noindent
We consider $\Ti$ as domain of functions $f:=(f^1,f^2):\Ti\rightarrow \C$, such that $f^j:e_j\rightarrow \C$ with $j=1,2$. Let $\Hi=L^2(\Ti,\C)$ be the Hilbert space equipped with the norm $\|\cdot\|$ induced by the scalar product
$$\la\psi,\ffi\ra:=\la\psi,\ffi\ra_{\Hi}=\int_{e_1}\overline{\psi^1}(x)\ffi^1(x)dx+\int_{e_2}\overline{\psi^2}(x)\ffi^2(x)dx,\ \ \ \  \ \ \forall \psi,\ffi\in\Hi.$$
For $s>0$, we introduce the spaces $H^s:=H^s(\Ti,\C)= H^s(e_1,\C)\otimes H^s(e_2,\C)$  and the bilinear Schr\"odinger equation in $\Hi$
\begin{equation}\label{mainT}\tag{BSE*}\begin{split}
\begin{cases}
i\dd_t\psi(t,x)=-\Delta\psi(t,x)+u(t)B\psi(t,x),\ \ \ \ \ \ \ \ &t\in(0,T), \ T>0,\\
\psi(0,x)=\psi_0(x),\ &\ \ \ \ \  \ \ \ \ \ \ \ \ \ \ x\in \Ti.\\
\end{cases}
\end{split}
\end{equation}
The Laplacian $-\Delta$ is equipped with self-adjoint boundary conditions as $v$ is equipped with {\it Neumann-Kirchhoff} boundary conditions, {\it i.e.}
\begin{equation*}\begin{split}
f\text{ is continuous in }v,\ \ \ \ \ \ \ \
\frac{\dd f^1}{\dd x}(0)-\frac{\dd f^1}{\dd x}(1)+\frac{\dd f^2}{\dd x}(0)=0
\end{split}
\end{equation*}
for every $f\in D(-\Delta)$. We assume $B:\psi\rightarrow (\mu\psi^1,0)$ with $\mu(x)=x(1-x)$ and $u\in L^2((0,T),\R)$. We call $\G_t^u$ the unitary propagator generated by the operator $$-\Delta+u(t)B.$$ 
The (\ref{mainT}) corresponds to the following Cauchy systems respectively in $L^2(e_1,\C)$ and $L^2(e_2,\C)$ with $t\in(0,T)$ and $T>0$
\begin{equation*}\begin{split}
\begin{cases}
i\dd_t\psi^1(t)=-\Delta\psi^1(t)+u(t)\mu\psi^1(t),\\
\psi^1(0)=\psi^1_0,\\
\end{cases}\ \ \ \ \ \begin{cases}
i\dd_t\psi^2(t)=-\Delta\psi^2(t),\\
\psi^2(0)=\psi_0^2.\\
\end{cases}
\end{split}
\end{equation*}
Let $\upvarphi:=\{\ffi_k\}_{k\in\N^*}$ be an orthonormal system of $\Hi$ made by eigenfunctions of $-\Delta$ and corresponding to the eigenvalues $\upmu:=\{\mu_k\}_{k\in\N^*}$ such that
\begin{equation*}\ffi_k=\big(\sqrt{{2}}\sin({2k}\pi x),0\big),\ \ \ \ \ \ \ \ \ \ \mu_k={4k^2\pi^2},\ \ \ \ \ \ \ \ \ \ \forall k\in \N^*.\end{equation*} 
We define $\Hi(\upvarphi):=\overline{span\{\varphi_k\ |\ k\in\N^*\}}^{\ L^2}$ and, for $s>0$, the spaces \begin{equation}\label{HSTspaces}
H^s_{\Ti}(\upvarphi)=\{\psi\in \Hi(\upvarphi)\ | \ \sum_{k\in\N^*}|k^s\la\ffi_k,\psi\ra|^2<\infty\}\end{equation}
equipped with the norms $\|\cdot\|_{(s)}=\big(\sum_{k\in\N^*}|k^s\la\ffi_k,\cdot\ra|^2\big)^{{1}/{2}}.$

	\subsection{Well-posedness}
\begin{prop}\label{lauraT}
	
	Let $\psi_0\in H^{3}_{\Ti}(\upvarphi)$ and $u\in L^2((0,T),\R)$. There
	exists a unique mild solution of the (\ref{mainT}) in
	$H^{3}_{\Ti}(\upvarphi)$, i.e. a function $\psi$ such that
	\begin{equation}\label{mild}
	\psi(t,x)=e^{i\Delta t}\psi_0(x)-
	i\int_0^t e^{i\Delta(t-s)}u(s)B\psi(s,x)ds \in C_0([0,T],H^{3}_{\Ti}(\upvarphi)).\\
	\end{equation}
	Moreover, there exists $C=C(T,B,u)>0$ so that $\|\psi\|_{C^0([0,T],H^{3}_{\Ti}(\upvarphi))}\leq C\|\psi_0\|_{(3)}$, while $\|\psi(t)\| =\|\psi_0\|$ for every $t\in[0,T]$ and $\psi_0\in H^{3}_{\Ti}(\upvarphi).$
\end{prop}
\begin{proof} 
The statement is proved by using the techniques developed in the proof of \cite[Proposition\ 4.1]{mio3}, which 
generalize the ones of \cite[Lemma\ 1;\ Proposition\ 2]{laurent}.
	\smallskip 
	
	\needspace{3\baselineskip}

	\noindent
	{\bf 1) } Let $\psi\in C^0([0,T], H^3_{\Ti}(\upvarphi))$. We notice $B\psi(s) \in H^{3}\cap H_{\Ti}^2(\upvarphi)$ for almost every $s\in (0,t)$ and $t\in(0,T)$. Let $G(t)=\int_0^t e^{i\Delta(t-s)}u(s)B\psi(s,x)ds$ so that
	\begin{equation*}\begin{split}
	\|G(t)\|_{(3)}&=\Big(\sum_{k\in\N^*}\Big|k^3\int_0^t e^{i\mu_ks}\la\ffi_k,u(s)B\psi(s,\cdot)\ra ds\Big|^2\Big)^\frac{1}{2}.\\
	\end{split}\end{equation*}
	We prove $G(\cdot)\in C^0([0,T], H^3_{\Ti}(\upvarphi))$. For $f(s,\cdot):=u(s)B\psi(s,\cdot)$ such that $f=(f^1,f^2)$, 
	\begin{equation*}\begin{split}
	&\la\ffi_k,f(s,\cdot)\ra	=
	\frac{1}{ \mu_k}\int_{\Ti} \ffi_k(y)\dd_{x}^2f(s,y)dy	=
	\frac{\sqrt{2}}{ (2k)^2 \pi^2}\int_{e_1}\sin(2k\pi y)\dd_{x}^2f^1(s,y)dy \\
	& =
	-\frac{\sqrt{2}}{(2k)^3 \pi^3 }\Big(\dd_{x}^2f^1(s,0)-\dd_{x}^2f^1(s,1)-\int_{e_1}\cos(2k\pi y) \dd_{x}^3f(s,y)dy\Big).\\
	\end{split}\end{equation*}
	Now, there exists $C_1>0$ so that
	\begin{equation*}\begin{split}
	&\left|k^3\int_0^t e^{i\mu_ks}\la\ffi_k,f(s)\ra ds\right|
	\leq{C_1}\left(\left|\int_0^t e^{i\mu_ks}\dd_{x}^2f^1(s,0)ds\right|\right.\\
	&\left.+
	\left|\int_0^t e^{i\mu_ks}\dd_{x}^2f^1(s,1)ds\right|+ \left|\int_0^t e^{i\mu_ks}\int_{e_1}\cos(2k\pi y) \dd_{x}^3f(s,y)dyds\right|\right).\\
	\end{split}\end{equation*}
		We notice $\dd_{x}^3f^1(s,\cdot)\in\overline{span\{\sqrt{{2}}\cos({ 2k\pi }x\big):\ k\in\N^*\}}^{ L^2}$ for almost every $s\in (0,t)$ and $t\in(0,T)$. Thus, 
	\begin{equation*}\begin{split}
	&\|G(t)\|_{(3)}\leq C_1\left(\Big\|\int_0^t \dd_{x}^2f^1(s,0)e^{i \mu_{(\cdot)}s}ds \Big\|_{\ell^2}+\Big\|\int_0^t \dd_{x}^2f^1(s,1)e^{i \mu_{(\cdot)}s}ds \Big\|_{\ell^2}\right.\\
	&+\left.\Big\|\int_0^t e^{i\mu_{(\cdot)}s}\int_{e_1}\cos(\sqrt{\mu_{(\cdot)}}y) \dd_{x}^3f(s,y)dyds\Big\|_{\ell^2}\right)\\
	&\leq C_1\left(\Big\|\int_0^t \dd_{x}^2f^1(s,0)e^{i \mu_{(\cdot)}s}ds \Big\|_{\ell^2}+\Big\|\int_0^t \dd_{x}^2f^1(s,1)e^{i \mu_{(\cdot)}s}ds \Big\|_{\ell^2}\right.\\
	&+\left.\sqrt{t}\left(\int_0^t\Big\|\int_{e_1}\cos(\sqrt{\mu_{(\cdot)}}y) \dd_{x}^3f(s,y)dy\Big\|_{\ell^2}^2ds\right)^\frac{1}{2}\right)	\\
	&\leq C_1\left(\Big\|\int_0^t \dd_{x}^2f^1(s,0)e^{i \mu_{(\cdot)}s}ds \Big\|_{\ell^2}+\Big\|\int_0^t \dd_{x}^2f^1(s,1)e^{i \mu_{(\cdot)}s}ds \Big\|_{\ell^2}+\sqrt{t}\|f\|_{L^2((0,t),H^3)}\right).	
	\end{split}\end{equation*}
	From \cite[Proposition\ B.6]{mio3}, there exist $C_2(t),C_3(t)>0$ uniformly bounded for $t$ in bounded intervals such that
	\begin{equation*}\begin{split}
	\|G(t)\|_{(3)} 
	&\leq C_2(t)\Big(\|\dd_{x}^2 f^1(\cdot,0)\|_{L^2((0,t),\C)}+\|\dd_{x}^2 f^1(\cdot,1)\|_{L^2((0,t),\C)}\Big)+\sqrt{t}\|f\|_{L^2((0,t),H^3)}
	\end{split}\end{equation*}
	and $\|G(t)\|_{(3)}\leq C_3(t)\|f(\cdot,\cdot)\|_{L^2((0,t),H^3)}$. For every $t\in [0,T]$, the last inequality shows that $G(t)\in H^3_\Ti(\upvarphi)$ and the provided upper bound is uniform. The Dominated Convergence Theorem leads to $G\in C^0([0,T], H^3_{\Ti}(\upvarphi))$.

	\smallskip
	
	\needspace{3\baselineskip}

	\noindent
	{\bf 2)} As $Ran(B|_{H_{\Ti}^{3}(\upvarphi)})\subseteq H^{3}\cap H^2_{\Ti}(\upvarphi)\subseteq H^{3}$, we have $B\in L(H^{3}_{\Ti}(\upvarphi),H^{3})$ thanks to the arguments of \cite[Remark\ 2.1]{mio1}. Let $\psi^0\in H_{\Ti}^{3}(\upvarphi)$. We consider the map $F:\psi \in C^0([0,T],H^{3}_{\Ti}(\upvarphi))\mapsto\phi\in C^0([0,T],H^{3}_{\Ti}(\upvarphi)) $ with
	$$\phi(t)=F(\psi)(t)=e^{i\Delta t}\psi^0-\int_0^te^{i\Delta(t-s)}u(s) B\psi(s)ds,\ \ \ \ \ \forall t\in [0,T].$$
	For every $\psi_1,\psi_2\in C^0([0,T],H^{3}_{\Ti}(\upvarphi))$, from the first point of the proof, there exists $C(t)>0$ uniformly bounded for $t$ lying on bounded intervals, such that
	\begin{equation*}\begin{split}
	&\|F(\psi_1)(t)-F(\psi_2)(t)\|_{({3})}\leq\left\|\int_0^t e^{i\Delta(t-s)}u(s) B(\psi_1(s)-\psi_2(s))ds\right\|_{({3})}\\
	&\leq C(t)\|u\|_{L^2((0,t),\R)}\iii B\iii_{L(H^{3}_{\Ti},H^{3})} \|\psi_1-\psi_2\|_{L^\infty((0,t),H^{3}_\Ti(\upvarphi))}.\\
	\end{split}\end{equation*}
	If $\|u\|_{L^2((0,t),\R)}$ is small enough, then $F$ is a contraction and Banach Fixed Point Theorem implies that there exists $\psi \in C^0([0,T],H^{3}_{\Ti}(\upvarphi)) $ such that $F(\psi)=\psi.$ When $\|u\|_{L^2((0,t),\R)}$ is not sufficiently small, one considers $\{t_j\}_{0\leq j\leq n}$ a partition of $[0,t]$ with $n\in\N^*$. We choose a partition such that each $\|u\|_{L^2([t_{j-1},t_j],\R)}$ is so small that the map $F$, defined on the interval $[t_{j-1},t_j]$, is a contraction and we apply the Banach Fixed Point Theorem. 
	
	\smallskip
	
	In conclusion, if $u\in C^0((0,T),\R)$, then $\psi\in C^1((0,T),\Hi(\upvarphi))$. By multiplying (\ref{mainT}) with $\psi(t)$, we obtain that $\dd_t\|\psi(t)\|^2=0$, which leads to $\|\psi(t)\| =\|\psi_0\|$ for every $t\in[0,T]$ and $\psi_0\in H^{3}_{\Ti}(\upvarphi)$. The generalization for $u\in L^2((0,T),\R)$ follows from a classical density argument. \qedhere
\end{proof}
\subsection{Global exact controllability}
\

\noindent
We recall that $\{\ffi_k\}_{k\in\N^*}$ and $\{\mu_k\}_{k\in\N^*}$ respectively are an orthonormal system of $\Hi$ made by eigenfunctions of $-\Delta$ and the corresponding eigenvalues. They are such that
\begin{equation*}\ffi_k=\big(\sqrt{{2}}\sin({2k}\pi x),0\big),\ \ \ \ \ \ \ \ \ \ \mu_k={4k^2\pi^2},\ \ \ \ \ \ \ \ \ \ \forall k\in \N^*.\end{equation*}
Let $\G_t^u$ be the unitary propagator representing the dynamics of (\ref{mainT}) at time $t\in[0,T]$ for $T>0$ and with control $u\in L^2((0,T),\R)$.

\begin{teorema}\label{globalegirino}
 The (\ref{mainT}) is globally exactly controllable in $H^3_\Ti(\upvarphi)$, {\it i.e.} for every $\psi_1,\psi_2\in H^3_\Ti(\upvarphi)$ such that $\|\psi_1\|=\|\psi_2\|$, there exist $T>0$ and $u\in L^2((0,T),\R)$ such that
	$$\G_T^u\psi_1=\psi_2.$$
	 In addition, the (\ref{mainT}) is energetically controllable in $\{\mu_k\}_{{k\in\N^*}}$, {i.e.,} for any  $m$ and $n\in\N^*$, there exist $T>0$ and $u\in L^2((0,T),\R)$ such that $$\G_T^u\varphi_m=\varphi_n.$$
	\end{teorema}
\begin{proof}
	{\bf 1) Local exact controllability in $H^3_{\Ti}(\upvarphi)$.} For $\epsilon,T,s>0$, let
	$$O_{\epsilon,T}^{s}:=\big\{\psi\in H_{\Ti}^{s}(\upvarphi)\big|\ \|\psi\|=1,\ \|\psi -\ffi_1(T)\|_{(s)}<\epsilon\big\}, \ \ \ \ \ \ \ \ffi_1(T)=e^{-i\mu_1T}\ffi_1.$$ We prove the existence of $T,\epsilon>0$ so that, for every $\psi\in O_{\epsilon,T}^{3}$, there exists $u\in L^2((0,T),\R)$ such that $\psi= \G^u_T\ffi_1.$
	To this purpose, we consider the map $\alpha$, the sequence with elements $\alpha_{k}(u)=\la \ffi_k(T), \G_{T}^u\ffi_1\ra$ for $k\in\N^*$, such that
	$$\alpha:L^2((0,T),\R)\longrightarrow Q:=\{{\bf x}:=\{x_k\}_{k\in\N^*}\in h^3(\C)\ |\ \|{\bf x}\|_{\ell^2}=1\}$$ 
	with $h^3$ defined in (\ref{spaces}). The local exact controllability of the bilinear Schr\"odinger equation in $O_{\epsilon,T}^{3}$ with $T>0$ is equivalent to the surjectivity of the map
	$\G_T^{(\cdot)}\ffi_1:u\in L^2((0,T),\R)\longmapsto \psi \in O_{\epsilon,T}^{s}\subset  H^3_{\Ti}(\upvarphi)$. 
As $$\G_{t}^u\ffi_1=\sum_{k\in\N^*}{\ffi_k(t)}\la \ffi_k(t),\G_{t}^u\ffi_1\ra, \ \ \ \ \  \ \ \ \ T>0,\ \ u\in 
L^2((0,T),\R),$$ the controllability is equivalent to the local surjectivity of $\alpha$. To this end, we use the 
Generalized Inverse Function Theorem (\cite[Theorem\  1;\ p.\ 240]{Inv}) and we study the surjectivity of 
$\gamma(v):=(d_u\alpha(0))\cdot\ v$ the Fréchet derivative of $\alpha$ with 
$\alpha(0)=\updelta=\{\delta_{k,1}\}_{k\in\N^*}$. Let $B_{j,k}:=\la\ffi_j,B\ffi_k\ra$ with $j,k\in\N^*$. As in 
the proof of \cite[Proposition\ 2.1]{mio2}, the map $\gamma$ is the sequence of elements $\gamma_{k}(v):=
	-i\int_{0}^Tv(\tau)e^{i(\mu_k-\mu_1)s}d\tau B_{k,1}$ with $k\in\N^*$ so that 
	$$\gamma:L^2((0,T),\R)\longrightarrow T_{\updelta}Q=\{{\bf x}:=\{x_k\}_{k\in\N^*}\in h^3(\C)\ |\ ix_1\in\R\}.$$ 
	The surjectivity of $\gamma$ corresponds to the solvability of the moments problem
	\begin{equation}\begin{split}\label{mome1}
	{x_{k}}/{B_{k,1}}=-i\int_{0}^Tu(\tau)e^{i(\mu_k-\mu_1)\tau}d\tau,\ \ \ \ \ \ \ \ \ \ \forall \{x_{k}\}_{k\in\N^*}\in T_{\updelta}Q\subset h^3. 
	\\
	\end{split}\end{equation} 
	By direct computation, we know $|\la\ffi_1,B\ffi_1\ra|\neq 0$ and, for $k\in\N^*\setminus\{1\}$, there holds 
	$$\la\ffi_k,B\ffi_1\ra=\int_0^1x(1-x)2\sin(2\pi x)\sin(2 k\pi x) ds=\frac{-2 k  }{ (k^2-1)^2 \pi^2}.$$
	Thus, there exists $C>0$ such that $|\la\ffi_k,B\ffi_1\ra|\geq C k^{-3}$ for every $k\in\N^*.$ Now, 
	$$\big\{x_k (\la\ffi_k,B\ffi_1\ra)^{-1}\big\}_{k\in\N^*}\in \ell^2,\ \ \ \ \ \ \ \ \ \ \ \ \ \ \ i{x_{1}}/\la\ffi_1,B\ffi_1\ra\in\R.$$ 
	In conclusion, the solvability of $(\ref{mome1})$ is guaranteed by \cite[Proposition\ B.5]{mio3} since $$\{x_k 
B_{k,1}^{-1}\}_{k\in\N^*}\in\{\{c_k\}_{k\in\N^*}\in \ell^2\ |\ c_1\in\R\},\ \ \ \ \ \ 
\inf_{k\in\N^*}|\mu_{k+1}-\mu_k|={12\pi^2}.$$

	\smallskip
	
	\noindent
	{\bf 2) Global exact controllability.}  Let $T,\epsilon>0$ be so that {\bf 1)} is valid. Thanks to Remark \ref{approxT} (Appendix \ref{approximatecon}), for any $\psi_1,\psi_2\in H^{3}_{\Ti}(\upvarphi)$ such that $\|\psi_1\|=\|\psi_2\|=p$, there exist $T_1,T_2>0$, $u_1\in L^2((0,T_1),\R)$ and $u_2\in L^2((0,T_2),\R)$ such that $$\|\G^{u_1}_{T_1}p^{-1}\psi_1-\ffi_1\|_{(3)}<{\epsilon},\ \ \ \ \ \|\G^{u_2}_{T_2}p^{-1}\psi_2-\ffi_1\|_{(3)}<{\epsilon}$$
	and $ p^{-1}\G^{u_1}_{T_1}\psi_1,p^{-1}\G^{u_2}_{T_2}\psi_2\in O_{\epsilon,T}^{3}.$ From {\bf 1)}, there exist $u_3,u_4\in L^2((0,T),\R)$ such that $$\G_T^{u_3}\G^{u_1}_{T_1}\psi_1=\G_T^{u_4}\G^{u_2}_{T_2}\psi_2=p\ffi_1.$$ In conclusion, there exist $T>0$ and $\widetilde u\in L^2((0,\widetilde T),\R)$ such that
	\begin{equation*}\G_{\widetilde T}^{\widetilde u}\psi_1=\psi_2.\end{equation*}

	\noindent
	{\bf 3) Energetic controllability.} The energetic controllability follows as $\ffi_k\in H^s_{\Ti}(\upvarphi)$ for every $s>0$ and $k\in\N^*.$ \qedhere
\end{proof}

\section{Generic graphs}\label{preli}

Let $\Gi$ be a generic infinite graph composed by $N\in\N^*\cup\{+\infty\}$ edges $\{e_j\}_{j\leq N}$ of lengths $\{L_j\}_{j\leq N}\subset\R^+\cup\{+\infty\}$ and $M\in\N^*$ vertices $\{v_j\}_{j\leq M}$.

\noindent
Let the bilinear Schr\"odinger equation in the Hilbert space $\Hi:=L^2(\Gi,\C)$
\begin{equation}\label{mainx1}\tag{BSE}\begin{split}
\begin{cases}
i\dd_t\psi(t,x)=-\Delta\psi(t,x)+u(t)B\psi(t,x),\ \ \ \ \ \ \ \ &t\in(0,T), \ T>0,\\
\psi(0,x)=\psi_0(x),\ &\ \ \ \ \  \ \ \ \ \ \ \ \ \ \ x\in \Gi.\\
\end{cases}
\end{split}
\end{equation}
The Laplacian $A=-\Delta$ is equipped with self-adjoint boundary conditions, $B$ is a bounded symmetric operator and $u\in L^2((0,T),\R)$. When the (\ref{mainx1}) is well-posed, we call $\G_t^u$ the unitary propagator generated by $A+u(t)B.$
We call $V_e$ and $V_i$ the external and the internal vertices of $\Gi$, {\it i.e.}
$$V_e:=\big\{v\in\{v_j\}_{ j\leq M}\ |\ \exists ! e\in\{e_j\}_{j\leq N}: v\in e\big\},\ \ \ \ \ V_i:=\{v_j\}_{ j\leq M}\setminus V_e.$$
For every $v$ vertex of $\Gi$, we denote $N(v):=\big\{l \in\{1,...,N\}\ |\ v\in e_l\big\}$ and each $e_k$ is considered to be parametrized with a coordinate going from $0$ to $L_k$. We equip $\Hi=L^2(\Gi,\C)$ with the scalar product
$$\la\psi,\ffi\ra:=\la\psi,\ffi\ra_{\Hi}=\sum_{j\leq N}\la\psi^j,\ffi^j\ra_{L^2(e_j,\C)}=\sum_{j\leq N}\int_{e_j}\overline{\psi^j}(x)\ffi^j(x)dx,\ \ \ \  \ \ \forall \psi,\ffi\in\Hi.$$
We call $\|\cdot\|=\sqrt{\la\cdot,\cdot\ra}$ the norm in $\Hi$ and, for $s>0$, we introduce the spaces
$$H^s:=H^s(\Gi,\C)=\Big\{\psi=(\psi^1,...,\psi^N)\in\prod_{j\leq N}H^s(e_j,\C)\ |\ \sum_{j\leq N}\|\psi^j\|_{H^s(e_j,\C)}^2<\infty\Big\}.$$ 
In the (\ref{mainx1}), the operator $A$ is a self-adjoint Laplacian such that the functions in $D(A)$ satisfy the following boundary conditions. Each $v\in V_i$ is equipped with {\it Neumann-Kirchhoff} boundary conditions when the function $f$ is continuous in $v$ and
\begin{equation*}\begin{split}
	\sum_{e\ni v}\frac{\dd f}{\dd x_e}(v)=0,\ \ \ \ \ \ \ \ \  \ \forall f\in D(A).
	\end{split}
	\end{equation*}
The derivatives are assumed to be taken in the directions away from the vertex (outgoing directions). In addition, the external vertices $V_e$ are equipped with {\it Dirichlet} or {\it Neumann} type boundary conditions. As in \cite{mio3}, we respectively call ($\NN\KK$), ($\Di$) and ($\NN$)  the {\it Neumann-Kirchhoff}, {\it Dirichlet} and {\it Neumann} boundary conditions characterizing $D(A)$.

\smallskip

In the current work, we denote a graph $\Gi$ as {\it quantum graph} when a self-adjoint Laplacian $A$ is defined on $\Gi$. 
We say that $\Gi$ is equipped with one of the previous boundaries in a vertex $v$, when each $f\in D(A)$ satisfies it in $v$.
By simplifying the notation of \cite{mio3}, we say that $\Gi$ is equipped with ($\Di$) (or ($\NN$)) when, for every $f\in D(A)$, the function $f$ satisfies ($\Di$) (or ($\NN$)) in every $v\in V_e$ and verifies ($\NN\KK$) in every $v\in V_i$. In addition, the graph $\Gi$ is equipped with ($\Di$/$\NN$) when, for every $f\in D(A)$ and $v\in V_e$, the function $f$ satisfies ($\Di$) or ($\NN$) in $v$ and $f$ verifies ($\NN\KK$) in every $v\in V_i$. 

\smallskip

Let $\upvarphi:=\{\ffi_k\}_{k\in\N^*}$ be an orthonormal system of $\Hi$ made by eigenfunctions of $A$ and let $\{\mu_{k}\}_{k\in\N^*}$ be the corresponding eigenvalues. We define  
$$\Gi(\upvarphi)=\bigcup_{k\in\N^*}supp(\ffi_k),\ \ \ \ \  \ \Hi(\upvarphi):=\overline{span\{\varphi_k\ |\ k\in\N^*\}}^{\ L^2},$$ $$H^s_{\Gi}(\upvarphi)=\{\psi\in \Hi(\upvarphi)\ | \ \sum_{k\in\N^*}|k^s\la\ffi_k,\psi\ra|^2<\infty\},\ \ \ \ \ \ \|\cdot\|_{(s)}^2=\sum_{k\in\N^*}|k^s\la\ffi_k,\cdot\ra|^2$$
with $s>0$. Let $V_e(\upvarphi)$ ($V_i(\upvarphi)$) be the external (internal) vertices of $\Gi(\upvarphi)$.
	\begin{osss}\label{normequivalence}
	Let $c\in \R^+$ be such that $0\not\in\sigma(A+c,\Hi(\upvarphi))$ (the spectrum of $A+c$ in the Hilbert space 
$\Hi(\upvarphi)$).  As $\Gi(\upvarphi)$ is a compact graph, thanks to \cite[Lemma\ 2.3]{mio3}, for every $s>0,$ we 
have $\|\cdot\|_{(s)}\asymp \||A+c|^\frac{s}{2}\cdot\|$ in $H^s_\Gi(\upvarphi)$, {\it i.e.} there exists 
$C_1,C_2>0$ such that $$C_1\|\psi\|_{(s)}\leq \||A+c|^{s/2}\psi\|\leq C_2\|\psi\|_{(s)},\ \ \ \ \ \ \ \ \forall 
\psi \in H^s_\Gi(\upvarphi).$$
\end{osss}

\noindent
Now, $\Gi(\upvarphi)$ is the quantum graph associated to a Laplacian $-\Delta$ so that
$$D(-\Delta)=\{\psi\in L^2(\Gi(\upvarphi),\C)\ |\ \exists \psi_1\in H^2_\Gi(\upvarphi)\ :\ \psi_1|_{\Gi(\upvarphi)}=\psi\}.$$
Let $[r]$ be the entire part of $r\in\R$. For $s>0$, we define the spaces
\begin{equation}\label{spaces}
		\begin{split}
		H^s_{\NN\KK}(\upvarphi):=\Big\{&\psi\in \Hi(\upvarphi)\cap H^s\ |\ \dd_x^{2n_2}\psi\text{ continuous in }v,\  \sum_{e\in N(v)}\dd_{x_e}^{2n_1+1}\psi(v)=0,\\
		& \ \forall n_1,n_2\in\N^*\cup\{0\},\ n_1<\big[({s+1})/{2}\big],\  n_2<\big[{s}/{2}\big],\  \forall v\in V_i\Big\},\\
		h^s:=\Big\{&\{a_k\}_{k\in\N^*}\subset{\C}\ \big|\ \sum_{k\in\N^*}|k^{s}a_k|^2<\infty\Big\}.\\
		\end{split}
		\end{equation}
		We equip the space $h^s$ for $s>0$ with the norm $\|\cdot\|_{(s)}$ such that
		$$\forall \{a_{k}\}_{k\in\N^*}\in h^s\ \ \ \ \ \big\|\{a_{k}\}_{k\in\N^*}\big\|_{(s)}:=\Big(\sum_{k\in\N^*}|k^sa_{k}|^2\Big)^{\frac{1}{2}}.$$
		Let $\eta>0,$ $a\geq 0$ and $I:=\{(j,k)\in(\N^*)^2:j\neq k\}.$		
		
\needspace{3\baselineskip}
\begin{assumptionI}[$\upvarphi,\eta$]
	The operator $B:\Hi(\upvarphi)\rightarrow\Hi(\upvarphi)$ is bounded and symmetric in $\Hi(\upvarphi)$, $Ran(B|_{H^2_{\Gi}(\upvarphi)})\subseteq H^2_{\Gi}(\upvarphi)$.
	
	\begin{enumerate}
		\item There exists $C>0$ such that $|\la\ffi_k,B\ffi_1\ra|\geq\frac{C}{k^{2+\eta}}$ for every $k\in\N^*$.
		\item For every $(j,k),(l,m)\in I$ such that $(j,k)\neq(l,m)$ and $\mu_j-\mu_k=\mu_l-\mu_m,$ it holds $\la\ffi_j,B\ffi_j\ra-\la\ffi_k,B\ffi_k\ra-\la\ffi_l,B\ffi_l\ra+\la\ffi_m,B\ffi_m\ra\neq 0.$
	\end{enumerate}
	
\end{assumptionI}

\begin{assumptionII}[$\upvarphi,\eta,a$] Let one of the following points be satisfied.
	\begin{enumerate}
		
		\item When $\Gi(\upvarphi)$ is equipped with ($\Di$/$\NN$) and $a+\eta\in(0, 3/2)$, there exists 
		$d\in[\max\{a+\eta,1\},3/2)$ such that $Ran(B|_{H_{\Gi}^{2+d}(\upvarphi)})\subseteq H^{2+d}\cap H^2_{\Gi}(\upvarphi).$

		\item When $\Gi(\upvarphi)$ is equipped with ($\NN$) and $a+\eta\in(0, 7/2)$, there exist $d\in[\max\{a+\eta,2\},7/2)$ and $d_1\in(d,7/2)$ such that $Ran(B|_{H_{\Gi}^{2+d}(\upvarphi)})\subseteq H^{2+d}\cap H^{1+d}_{\NN\KK}(\upvarphi)\cap H^2_{\Gi}(\upvarphi)$ and $Ran(B|_{ H^{d_1}_{\NN\KK}(\upvarphi)})\subseteq H^{d_1}_{\NN\KK}(\upvarphi).$

		\item When $\Gi$ is equipped with ($\Di$) and $a+\eta\in(0, 5/2)$, there exists 
$d\in[\max\{a+\eta,1\},5/2)$ such that $Ran(B|_{H_{\Gi}^{2+d}(\upvarphi)})\subseteq H^{2+d}\cap 
H^{1+d}_{\NN\KK}(\upvarphi)\cap H^2_{\Gi}(\upvarphi).$ If $d\geq 2$, then there exists $d_1\in(d,5/2)$ such that 
there holds $Ran(B|_{H^{d_1}\cap \Hi(\upvarphi)})\subseteq H^{d_1}\cap \Hi(\upvarphi).$

	\end{enumerate}
\end{assumptionII}

From now on, we omit the terms $\upvarphi,$ $\eta$ and $a$ from the notations of Assumptions I and Assumptions II when their are not relevant.

\subsection{Interpolation properties and well-posedness}\label{well}

We present {\it interpolation properties} for the spaces $H^s_{\Gi}(\upvarphi)$ with $s>0$. The result follows from 
\cite[Proposition\ 4.2]{mio3} as $\Gi(\upvarphi)$ is a compact graph.
\begin{prop}[Proposition 4.2; \cite{mio3}]\label{bor}Let $\upvarphi:=\{\ffi_k\}_{k\in\N^*}$ be an orthonormal 
system of $\Hi$ made by eigenfunctions of $A$.
	
	\smallskip
	\noindent
	{\bf 1)} If the quantum graph $\Gi(\upvarphi)$ is equipped with ($\Di$/$\NN$), then
	$$H^{s_1+s_2}_{\Gi}(\upvarphi)=H_{\Gi}^{s_1}(\upvarphi)\cap H^{s_1+s_2} \ \ \ \text{for}\ \ \ s_1\in\N,\ s_2\in[0,1/2).$$
	
	\noindent
	{\bf 2)} If the quantum graph $\Gi(\upvarphi)$ is equipped with ($\NN$), then
	$$H^{s_1+s_2}_{\Gi}(\upvarphi)=H_{\Gi}^{s_1}(\upvarphi)\cap H^{s_1+s_2}_{\NN\KK}(\upvarphi) \ \ \ \text{for}\ \ \ s_1\in 2\N\,\ s_2\in[0,3/2).$$
	
	\noindent
	{\bf 3)} If the quantum graph $\Gi(\upvarphi)$ is equipped with ($\Di$), then
	$$H^{s_1+s_2+1}_{\Gi}(\upvarphi)=H_{\Gi}^{s_1+1}(\upvarphi)\cap H^{s_1+s_2+1}_{\NN\KK}(\upvarphi) \ \ \ \text{for}\ \ \ s_1\in 2\N,\ s_2\in[0,3/2).$$
\end{prop}

In the following section, we ensure the {\it well-posedness} of the (\ref{mainx1}).

\begin{prop}\label{laura}
	
	Let the couple $(A,B)$ satisfy Assumptions II$(\upvarphi,\eta,\tilde d)$ with $\eta>0$ and $\tilde d\geq 0$. Let $d$ be introduced in Assumptions II.
	
	\medskip
	\noindent
	{\bf 1)} Let $T >0$ and $ f\in L^2((0,T), H^{2+d}\cap H^{1+d}_{\NN\KK}(\upvarphi)\cap H^2_{\Gi}(\upvarphi)$. Let $t\mapsto G(t)=\int_0^te^{iA\tau} f(\tau) d\tau.$ The map $G\in C^0([0,T], H^{2+d}_{\Gi}(\upvarphi))$ and there exists $C(T)>0$ uniformly bounded for $T$ lying on intervals so that
	$$\|G\|_{L^{\infty}((0,T),H^{2+d}_{\Gi}(\upvarphi))}\leq C(T)\| f\|_{L^2((0,T),H^{2+d})}.$$
	
	\noindent
	{\bf 2)} Let $\psi_0\in H^{2+d}_{\Gi}(\upvarphi)$ and $u\in L^2((0,T),\R)$. There
	exists a unique mild solution $\psi\in C_0([0,T],H^{3}_{\Ti}(\upvarphi))$ of the (\ref{mainx1}) (relation $(\ref{mild})$).
	Moreover, there exists $C=C(T,B,u)>0$ so that, for every $t\in[0,T]$ and $\psi_0\in H^{2+d}_{\Gi}(\upvarphi)$, $$\|\psi\|_{C^0([0,T],H^{2+d}_{\Gi}(\upvarphi))}\leq C\|\psi_0\|_{(2+d)},\ \ \ \ \ \ \|\psi(t)\| =\|\psi_0\|.$$
\end{prop}
\begin{proof} The result is obtained by generalizing the proof of Proposition $\ref{lauraT}$.
	
	\smallskip 
	
	\needspace{3\baselineskip}

	\noindent
	{\bf 1) (a) Assumptions II.1 .} Let $f(s)\in H^{3}\cap H_{\Gi}^2(\upvarphi)$ for almost every $s\in (0,t)$, $t\in(0,T)$ and $f(s)=(f^1(s),...,f^N(s))$. We prove that $G\in C^0([0,T], H^3_{\Gi}(\upvarphi))$. First, $G(t)=\sum_{k=1}^\infty\ffi_k \int_0^t e^{i\mu_ks}\la\ffi_k,f(s)\ra ds$ and
	\begin{equation}\label{Gspectraldecomposition}\begin{split}
	\|G(t)\|_{(3)}&=\Big(\sum_{k\in\N^*}\Big|k^3\int_0^t e^{i\mu_ks}\la\ffi_k,f(s)\ra ds\Big|^2\Big)^\frac{1}{2}.\\
	\end{split}\end{equation}
	We estimate $\la\ffi_k,f(s,\cdot)\ra$ for each $k\in\N^*$ and $s\in(0,t)$. We suppose $\mu_1\neq 0$. Let $\dd_x f(s)=(\dd_x f^1(s),...,\dd_x f^N(s))$ be the derivative of $f(s)$ and $P(\ffi_k)=(P(\ffi^1_k),...,P(\ffi^N_k))$ be the primitive of $\ffi_k$ so that $P(\ffi_k)=-\frac{1}{\mu_k}\dd_{x}\ffi_k.$ We call $\dd e$ the two points of the boundaries of an edge $e$. For every $v\in V_e(\upvarphi)$, $\tilde v	\in V_i(\upvarphi)$ and $j\in N(\tilde v)$, there exist $a(v),a^j(\tilde v)\in\{-1,+1\}$ so that 
	\begin{equation}\label{computations}\begin{split}
	&\la\ffi_k,f(s)\ra	=
	\frac{1}{\mu_k }\int_\Gi\ffi_k(y)\dd_{x}^2f(s,y)dy	=
	\frac{1}{\mu_k^2 }\int_{\Gi(\upvarphi)}\dd_{x}\ffi_k(y)\dd_{x}^3f(s,y)dy\\
	&+\frac{1}{\mu_k^2}\sum_{v\in V_i(\upvarphi)} 
	\sum_{j\in N(v)}a^j(v)\dd_{x}\ffi_k^j(v)\dd_{x}^2f^j(s,v)+\frac{1}{\mu_k^2}\sum_{v\in V_e} 
	a(v)\dd_{x}\ffi_k(v)\dd_{x}^2f(s,v).\\
	\end{split}\end{equation}
	We consider \cite[Lemma  \ 2.3]{mio3} since $\Gi(\upvarphi)$ is a compact graph. There exist $C_1>0$ such that 
$\mu_k^{-2}\leq C_1 k^{-4}$ for every $k\in\N^*$ and
	\begin{equation}\label{computations1}\begin{split}
	&\left|k^3\int_0^t e^{i\mu_ks}\la\ffi_k,f(s)\ra ds\right|
	\leq\frac{C_1}{k}\left(\sum_{v\in V_e(\upvarphi)} 
	\left|\dd_{x}\ffi_k(v)\int_0^t e^{i\mu_ks}\dd_{x}^2f(s,v)ds\right|\right.\\
	&+\sum_{v\in V_i(\upvarphi)} 
	\sum_{j\in N(v)}\left|\dd_{x}\ffi_k^j(v)\int_0^t e^{i\mu_ks}\dd_{x}^2f^j(s,v)ds\right|+
	\\
	& \left.\left|\int_0^t e^{i\mu_ks}\int_{\Gi(\upvarphi)}\dd_{x}\ffi_k(y)\dd_{x}^3f(s,y)dyds\right|\right).\\
	\end{split}\end{equation}
	\begin{osss}\label{Newselfadjointoperator}
		We notice $A'\mu_k^{-1/2}\dd_{x}\ffi_k=\mu_k\mu_k^{-1/2}\dd_{x}\ffi_k$ for every $k\in\N^*,$ where $A'=-\Delta$ is a self-adjoint Laplacian with compact resolvent. 
		Thus, $$\|\mu_k^{-1/2}\dd_{x}\ffi_k\|^2=\la\mu_k^{-1/2}\dd_{x}\ffi_k,\mu_k^{-1/2}\dd_{x}\ffi_k\ra=\la \ffi_k,\mu_k^{-1}A\ffi_k\ra=1$$ and, for almost every $s\in (0,t)$ and $t\in(0,T)$, $\dd_{x}^3f(s,\cdot)\in\overline{span\big\{\mu_k^{-1/2}\dd_{x}\ffi_k:\ k\in\N^*\big\}}^{ L^2}.$
	\end{osss}
	Let ${\bf a^l}=\{a_k^l\},{\bf b^l}=\{b_k^l\}\subset\C$ for $l\leq N$ be so that $\ffi_k^l(x)=a_k^l\cos(\sqrt{\mu_k}x)+ b_k^l\sin(\sqrt{\mu_k}x)$ and $-a_k^l\sin(\sqrt{\mu_k}x)+ b_k^l\cos(\sqrt{\mu_k}x)=\mu_k^{-1/2}\dd_{x}\ffi_k^l(x).$ Now, $$2\geq \|\mu_k^{-1/2}\dd_{x}\ffi_k^l\|_{L^2(e^l)}^2+\|\ffi_k^l\|_{L^2(e^l)}^2=(|a_k^l|^2+|b_k^l|^2)|e_l|$$ for every $k\in\N^*$ and $l\in\{1,...,N\}$. Thus, ${\bf a^l},{\bf b^l}\in\ell^\infty(\C)$ and there exists $C_2>0$ such that, for every $k\in\N^*$ and $v\in V_e\cup V_i$, we have $|\mu_k^{-1/2}\dd_{x}\ffi_k(v)|\leq C_2$. Thanks to the identities $(\ref{Gspectraldecomposition})$, $(\ref{computations1})$ and to Remark $\ref{Newselfadjointoperator}$, there exists $C_3>0$ such that
	\begin{equation}\begin{split}\label{upperboundG}
	\|G(t)\|_{(3)}&\leq C_3\sum_{v\in V_e(\upvarphi)\cup V_i(\upvarphi)}\sum_{j\in N(v)}\Big\|\int_0^t\dd_{x}^2 f^j(s,v)e^{i \mu_{(\cdot)}s}ds \Big\|_{\ell^2}\\
	&+C_3\Big\|\int_0^t\big\la {\mu_{(\cdot)}^{-1/2}}\dd_x\ffi_{(\cdot)}(s),\dd_{x}^3 f(s)\big\ra e^{i \mu_{(\cdot)}s}ds\Big\|_{\ell^2}.
	\end{split}\end{equation}
	Again, as $\Gi(\upvarphi)$ is a compact graph, \cite[Lemma\ 2.4]{mio3} is valid for the sequence $\upmu$ and, 
from \cite[Proposition\ B.6]{mio3}, there exist $C_4(t),C_5(t)>0$ uniformly bounded for $t$ in bounded intervals 
such that
	\begin{equation}\label{upperboundG1}\begin{split}
	\|G\|_{(3)} 
	&\leq C_4(t)\sum_{v\in V_e(\upvarphi)\cup V_i(\upvarphi)}\sum_{j\in N(v)}\|\dd_{x}^2 f^j(\cdot,v)\|_{L^2((0,t),\C)}+\sqrt{t}\|f\|_{L^2((0,t),H^3)}
	\end{split}\end{equation}
	and $\|G\|_{(3)}\leq C_5(t)\|f(\cdot,\cdot)\|_{L^2((0,t),H^3)}$. We underline that the identity is also valid when $\mu_1=0$, which is proved by isolating the term with $k=1$ and by repeating the steps above. 
	For every $t\in [0,T]$, the inequality (\ref{upperboundG1}) shows that $G(t)\in H^3_\Gi(\upvarphi)$. The provided upper bounds are uniform and the Dominated Convergence Theorem leads to $G\in C^0([0,T], H^3_{\Gi}(\upvarphi))$.

	\noindent
	Let $f(s)\in H^{5}\cap H_{\Gi}^4(\upvarphi)$ for almost every $s\in (0,t)$ and $t\in(0,T)$. The same techniques adopted above shows that $G\in C^0([0,T], H^5_{\Gi}(\upvarphi))$.

	\smallskip
	We denote $F(f)(t):=\int_0^{t}e^{iA\tau} f(\tau) d\tau$ for $f\in\Hi$ and $t\in (0,T)$. Let $X(B)$ be the space of functions $f$ so that $f(s)$ belongs to a Banach space $B$ for almost every $s\in (0,t)$ and $t\in(0,T)$. The first part of the proof implies
	$$F: X(H^{3}\cap H_{\Gi}^2(\upvarphi))\longrightarrow C^0([0,T], H^3_{\Gi}(\upvarphi)),$$ $$F: X(H^{5}\cap H_{\Gi}^4(\upvarphi))\longrightarrow C^0([0,T], H^5_{\Gi}(\upvarphi)).$$
	Classical interpolation results (as \cite[Theorem\ 4.4.1]{bergolo} with $n=1$) lead to $F: X(H^{2+d}\cap H_{\Gi}^{1+d}(\upvarphi))\longrightarrow C^0([0,T], H^{2+d}_{\Gi})$ with $d\in [1,3]$. Thanks to Proposition \ref{bor}, if $d\in [1,3/2)$ and $f(s)\in H^{2+d}\cap H^{1+d}_{\NN\KK}(\upvarphi)\cap H^2_{\Gi}(\upvarphi)=H^{2+d}\cap H^{1+d}_{\Gi}(\upvarphi)$ for almost every $s\in (0,t)$ and $t\in(0,T)$, then $G\in C^0([0,T], H^{2+d}_{\Gi}(\upvarphi))$, which achieves the proof.
	
	\smallskip 
	
	\needspace{3\baselineskip}

	\noindent
	{\bf (b) Assumptions II.3 .} If $\Gi(\upvarphi)$ is equipped with ($\Di$), then $H_{\Gi}^2(\upvarphi)= H_{\NN\KK}^2(\upvarphi)\cap H_{\Gi}^1(\upvarphi)$
	and $H_{\Gi}^4(\upvarphi)=H_{\NN\KK}^4(\upvarphi)\cap H_{\Gi}^3(\upvarphi)$ from Proposition $\ref{bor}$. As above, if $f(s)\in H^{3}\cap H_{\NN\KK}^2(\upvarphi)\cap H_{\Gi}^1(\upvarphi)$ for almost every $s\in (0,t)$ and $t\in(0,T)$, then $G\in C^0([0,T], H^3_{\Gi}(\upvarphi))$, while if $f(s)\in H^{5}\cap H_{\NN\KK}^4(\upvarphi)\cap H_{\Gi}^3(\upvarphi)$ for almost every $s\in (0,t)$ and $t\in(0,T)$, then $G\in C^0([0,T], H^5_{\Gi}(\upvarphi))$. 
	From the interpolation techniques, if $d\in [1,5/2)$ and $f(s)\in H^{2+d}\cap H^{1+d}_{\NN\KK}(\upvarphi)\cap H^{d}_{\Gi}(\upvarphi)$ for almost every $s\in (0,t)$ and $t\in(0,T)$, then $G\in C^0([0,T], H^{2+d}_{\Gi}(\upvarphi))$.

	\smallskip 
	
	\needspace{3\baselineskip}

	\noindent
	{\bf (c) Assumptions II.2 .} Let $f(s)\in H^{4}\cap H_{\NN\KK}^3(\upvarphi)\cap H_\Gi^2(\upvarphi)$ for almost every $s\in (0,t)$ and $t\in(0,T)$ and $\Gi(\upvarphi)$ be equipped with $(\NN)$. In this framework, the last line of $(\ref{computations})$ is zero. Indeed, $\dd_x^2f(s)\in C^0$ as $f (s)\in H_{\NN\KK}^3(\upvarphi)$ and, for $v\in V_e(\upvarphi)$, we have $\dd_{x}\ffi_k(v)=0$ thanks to the ($\NN$) boundary conditions (the terms $a^j(v)$ assume different signs according to the orientation of the edges connected in $v$). After, for every $v\in V_i(\upvarphi)$, thanks to the $(\NN\KK)$ in $v\in V_i(\upvarphi)$, we have $\sum_{j\in N(v)}a^j(v)\dd_{x}\ffi_k^j(v)=0$. 
	From $(\ref{computations})$, we obtain
	\begin{equation*}\begin{split}
	\la\ffi_k,f(s)\ra&=-\frac{1}{\mu_k^2 }\int_{\Gi(\upvarphi)}\dd_{x}\ffi_k(y)\dd_{x}^3f(s,y)dy=-\frac{1}{\mu_k^2}\sum_{v\in V_e(\upvarphi)} 
	a(v)\ffi_k(v)\dd_{x}^3f(s,v)\\
	&-\frac{1}{\mu_k^2}\sum_{v\in V_i(\upvarphi)} 
	\sum_{j\in N(v)}a^j(v)\ffi_k^j(v)\dd_{x}^3f^j(s,v)+
	\frac{1}{\mu_k^2 }\int_{\Gi(\upvarphi)}\ffi_k(y)\dd_{x}^4f(s,y)dy.\\
	\end{split}\end{equation*}
	Now, $\{\ffi_k\}_{k\in\N^*}$ is a Hilbert basis of $\Hi(\upvarphi)$ and we proceed as in (\ref{computations1}), (\ref{upperboundG}) and (\ref{upperboundG1}). From \cite[Proposition\ B.6]{mio3}, there exists $C_6(t)>0$ uniformly bounded such that $$\|G\|_{(4)}\leq C_1(t)\|f(\cdot,\cdot)\|_{L^2((0,t),H^4)}.$$

	\smallskip
	
	If $f(s)\in H^{4}\cap H_{\NN\KK}^3(\upvarphi)\cap H_\Gi^2(\upvarphi)$ for almost every $s\in (0,t)$ and $t\in(0,T)$, then $G\in C^0([0,T], H^{4}_{\Gi}(\upvarphi)).$ Equivalently when $f(s)\in H^{6}\cap H_{\NN\KK}^5(\upvarphi)\cap H_\Gi^4(\upvarphi)$ for almost every $s\in (0,t)$ and $t\in(0,T)$, we have $G\in C^0([0,T], H^{6}_{\Gi}(\upvarphi)).$
	As above, from Proposition $\ref{bor}$, if $d\in[2,7/2)$ and $f(s)\in H^{2+d}\cap H^{1+d}_{\NN\KK}(\upvarphi)\cap H^2_{\Gi}(\upvarphi)$ for almost every $s\in (0,t)$ and $t\in(0,T)$, then $G\in C^0([0,T], H^{2+d}_{\Gi}(\upvarphi)).$ 
	
	\medskip
	
	\needspace{3\baselineskip}

	\noindent
	{\bf 2)} As $Ran(B|_{H_{\Gi}^{2+d}(\upvarphi)})\subseteq H^{2+d}\cap H^{1+d}_{\NN\KK}(\upvarphi) H^2_{\Gi}(\upvarphi)\subseteq H^{2+d}$, we have \\
	$B\in L(H^{2+d}_{\Gi}(\upvarphi),H^{2+d})$ thanks to the arguments of \cite[Remark\ 2.1]{mio1}. Let $F:\psi 
\in C^0([0,T],H^{2+d}_{\Gi}(\upvarphi))\mapsto\phi\in C^0([0,T],H^{2+d}_{\Gi}(\upvarphi)) $ with
	$$\phi(t)=F(\psi)(t)=e^{-iAt}-\int_0^te^{-iA(t-s)}u(s) B\psi(s)ds,\ \ \ \ \ \forall t\in [0,T].$$
	For every $\psi_1,\psi_2\in H^{2+d}_{\Gi}(\upvarphi)$, from the first point of the proof, there exists $C(t)>0$ uniformly bounded for $t$ lying on bounded intervals, such that
	\begin{equation*}\begin{split}
	&\|F(\psi_1)(t)-F(\psi_2)(t)\|_{({2+d})}\leq\left\|\int_0^t e^{-iA(t-s)}u(s) B(\psi_1(s)-\psi_2(s))ds\right\|_{({2+d})}\\
	&\leq C(t)\|u\|_{L^2((0,t),\R)}\iii B\iii_{L(H^{2+d}_{\Gi},H^{2+d})} \|\psi_1-\psi_2\|_{L^\infty((0,t),H^{2+d}_\Gi(\upvarphi))}.\\
	\end{split}\end{equation*}
The proof is achieved as in the point {\bf 2.\,\,}of the proof of Proposition $\ref{lauraT}$.\qedhere
\end{proof}

\subsection{Controllability results}
\begin{defi}
Let $\upvarphi:=\{\ffi_k\}_{k\in\N^*}$ be an orthonormal system of $\Hi$ made by eigenfunctions of $A$ and let $\{\mu_{k}\}_{k\in\N^*}$ be the corresponding eigenvalues. 

\begin{enumerate}	
	\item The (\ref{mainx1}) is said to be globally exactly controllable in $H^s_\Gi(\upvarphi)$ with $s\geq 3$ if, for every $\psi_1,\psi_2\in H^s_\Gi(\upvarphi)$ such that $\|\psi_1\|=\|\psi_2\|$, there exist $T>0$ and $u\in L^2((0,T),\R)$ such that $\G_T^u\psi_1=\psi_2.$
	\item	The (\ref{mainx1}) is energetically controllable in $\{\mu_k\}_{{k\in\N^*}}$ if, for every $m,n\in\N^*$, there exist $T>0$ and $u\in L^2((0,T),\R)$ so that $\G_T^u\varphi_m=\varphi_n.$
	\end{enumerate}
\end{defi}
Before proceeding with the main result of the work, we notice the following fact. As $\Gi(\upvarphi)$ is a compact 
graph, \cite[Lemma\ 2.4]{mio3} implies
\begin{equation}\label{weakspectralgap}\begin{split}
\exists \MM\in\N^*,\ \delta>0\ \ :\ \ \ \inf_{{k\in\N^*}}|\mu_{k+\MM}-\mu_k|>\delta \MM\\
\end{split}\end{equation}
(the parameter $\MM$ is equal to $1$ when $\Gi(\upvarphi)$ corresponds to an interval). 
\begin{teorema}\label{globalenergetic}
	Let $\Gi$ be a quantum graph. We assume that 
	\begin{equation}\label{weakgap}\forall \epsilon>0,\ \exists C>0,\ \tilde d\geq 1\ \ :\ \ |\mu_{k+1}-\mu_k|\geq C  k^{-{\tilde d}-1},\ \ \forall k\in\N^*.\end{equation}
	If $(A,B)$ satisfies Assumptions I$(\upvarphi,\eta)$ and Assumptions II$(\upvarphi,\eta,\tilde d-1)$ for $\eta>0$, then the (\ref{mainx1}) is globally exactly controllable in $H^{s}_{\Gi}(\upvarphi)$ for $s=2+d$ with $d$ from Assumptions I and energetically controllable in $\{\mu_k\}_{{k\in\N^*}}.$
\end{teorema}
\begin{proof}
{\bf 1) Local exact controllability.} The proof follows as the point {\bf 1.\,\,}of the proof of Theorem $\ref{globalegirino}$ by considering $s=2+d$ instead of $s=3$. The peculiarity of this case is that $\alpha$ assumes value in $Q:=\{{\bf x}:=\{x_k\}_{k\in\N^*}\in h^s(\C)\ |\ \|{\bf x}\|_{\ell^2}=1\},$ while $\gamma$ in $$T_{\updelta}Q=\{{\bf x}:=\{x_k\}_{k\in\N^*}\in h^s(\C)\ |\ ix_1\in\R\}.$$ In the current framework, the moments problem $(\ref{mome1})$ is defined for sequences in $T_{\updelta}Q\subset h^s$ and $\big\{x_k (\la\ffi_k,B\ffi_1\ra)^{-1}\big\}_{k\in\N^*}\in h^{d-\eta}\subseteq h^{\tilde d-1}$ thanks to the point {\bf 1.\,\,}of Assumptions I. 
The solvability of $(\ref{mome1})$ is guaranteed by \cite[Proposition\ B.5]{mio3} thanks to $(\ref{weakgap})$ since
$$\{x_k B_{k,1}^{-1}\}_{k\in\N^*}\in\{\{c_k\}_{k\in\N^*}\in h^{\tilde d-1}(\C)\ |\ c_1\in\R\}.$$

\medskip
\noindent
{\bf 2) Global exact controllability and energetic controllability.} The proof is achieved as in the points {\bf 2.\,\,}and {\bf 3.\,\,}of the proof of Theorem \ref{globalegirino} by using Proposition $\ref{approx}$ (Appendix \ref{approximatecon}).\qedhere
\end{proof}

\section{Example}\label{exampleNOTtadpole}
Let a {\it star graph} be a graph composed by $N\in\N^*$ edges $\{e_j\}_{j\leq N}$. Each edge $e_j$ is \\ parametrized with a coordinate going from $0$ to the length of the edge $L_j$. We set the $0$ in the external vertex belonging to $e_j$.

\begin{figure}[H]
	\centering
	\includegraphics[width=\textwidth-50pt]{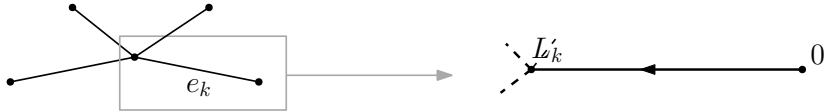}
	\caption{Parametrization of a star graph with $N=4$ edges.}\label{parametrizzazione}
\end{figure}

	Let $\Gi$ be a graph containing as sub-graph a star graph equipped with ($\Di$) and composed by the edges $\{e_j\}_{j\leq 4}$. Let the couple of edges $\{e_{1},e_{2}\}$ be of length $L_1=\sqrt[3]{2}$, while $\{e_{3},e_{4}\}$ be long $L_2=\sqrt[3]{5}$.

		\begin{figure}[H]
		\centering
		\includegraphics[width=\textwidth-50pt]{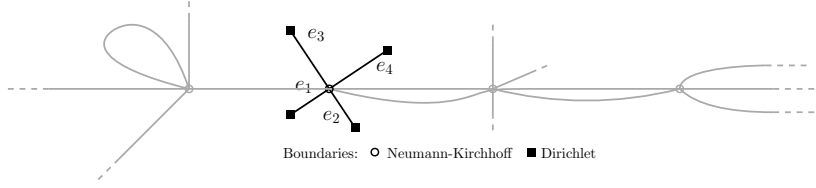}
		\caption{Example of star graph described in Section $\ref{exampleNOTtadpole}$.}
	\end{figure}
	
	\begin{coro}
		Let $B$ be such that $B\psi=((B\psi)^1,...,(B\psi)^N)$ for every $\psi\in\Hi$ and 
	\begin{equation*}
	\begin{split} 
	&(B\psi)^{1}=-(B\psi)^{2}=\sqrt[3]{2}\cos\Big(\frac{\pi x}{3 \sqrt[3]{2} }\Big)\psi^{1}(x)+\sqrt[3]{2}\cos\Big(\frac{\pi x}{3 \sqrt[3]{2}}\Big)\psi^{3}\Big(\frac{\sqrt[3]{5}}{\sqrt[3]{2}}x\Big),\\
	&(B\psi)^{3}=-(B\psi)^{4}=\sqrt[3]{5}\cos\Big(\frac{\pi x}{3 \sqrt[3]{5} }\Big)\psi^{3}(x)+\sqrt[3]{5}\cos\Big(\frac{\pi x}{3 \sqrt[3]{5}}\Big)\psi^{1}\Big(\frac{\sqrt[3]{2}}{\sqrt[3]{5}}x\Big),\\	\end{split}
	\end{equation*}
	while $(B\psi)^l\equiv 0$ for every $5\leq l\leq N$. There exists $\upvarphi:=\{\varphi_{k}\}_{k\in\N^*}$ an orthonormal system composed by eigenfunctions of $A$ such that the (\ref{mainx1}) is globally exactly controllable in $H^{3+\epsilon}_\Gi(\upvarphi)$ with $\epsilon>0$ and energetically controllable in $\big\{\frac{k^2\pi^2}{L_l}\big\}_{\underset{l\leq 2}{k,l\in\N^*}}.$ 
	
\end{coro}

\begin{proof}
	Let $\upvarphi=\{\ffi_k\}_{k\in\N^*}$ be some eigenfunctions of $A$ and $\upmu=\{\mu_{k}\}_{k\in\N^*}$ the corresponding eigenvalues. We define $\upvarphi$ and $\upmu$ so that, for every $k\in\N^*$, there exist $m(k)\in\N^*$ and $l(k)\in\{1,2\}$ so that $\ffi_k^{n}\equiv 0$ for $n\neq 2l(k),$ $2l(k)-1$ and
	$$\mu_k={m(k)^2\pi^2}{L^{-2}_{l(k)}},\ \  \ \ \ffi_k^{2l(k)-1}(x)=-\ffi_k^{2l(k)}(x)=\sqrt{{L_{l(k)}^{-1}}}\sin{(\sqrt{\mu_k} x)}.$$

	\noindent
	{\bf Spectral behaviour.} We notice that $\{1,\sqrt[3]{2},\sqrt[3]{5}\}$ are irrationally independent and 
$\frac{\sqrt[3]{2^2}}{\sqrt[3]{5^2}}$ is an algebraic irrational number. As in the proof of \cite[Lemma\ 
2.6]{mio3}, thanks \cite[Proposition\ A.1]{mio3}, for every $\epsilon>0,$ there exist $C>0$ and $\tilde d\geq 0$ 
such that
	$$|\mu_{k+1}-\mu_k|\geq C  k^{-{\tilde d}},\ \ \ \ \ \ \ \ \ \ \ \ \forall k\in\N^*.$$
	{\bf Assumptions I.1} For $[r]$ the entire part of $r\in\R^+$, we have
	\begin{equation*} \begin{split}
			&|\la\ffi_1, B\ffi_k\ra|=\Bigg|\sum_{l=1}^4\int_0^{L_{[(l+1)/2]}} \ffi_k^{l}(x)\sum_{n=1}^{2}{L_l}\cos\Big(\frac{\pi x}{3 L_{[(l+1)/2]}}\Big)\ffi_1^{2n-1}\Big(\frac{L_{n}}{L_{[(l+1)/2]}}x\Big)dx\Bigg|\\
			&=\Big|\int_0^{L_{l(k)}}2L_{l(k)}\cos\Big(\frac{\pi x}{3L_{l(k)}}\Big)\sin\Big(\frac{m(1)\pi x}{L_{l(k)}}\Big)\sin\Big(\frac{m(k)\pi x}{L_{l(k)}} \Big) dx\Big|\\
			&\geq {2^{5/3}}\Big|\int_0^1 \cos\Big(\frac{\pi x}{3}\Big)\sin(\pi x)\sin(m(k)\pi x) dx\Big|=\frac{3^3\,{2^{5/3}} \,\sqrt{3} m(k) }{ (64 - 
			180 m(k)^2 + 81 m(k)^4) \pi}.\\
			\end{split}\end{equation*}
			The last relation implies the existence of $C_1>0$ such that $\la\ffi_1,B\ffi_k\ra\geq C/k^3$ for every $k\in\N^*$ and the point {\bf 1.\,\,}of Assumptions I($\upvarphi,1$) is verified.
			
			\smallskip
			\noindent{\bf Assumptions I.2} We prove that the point {\bf 2.\,\,}of Assumptions I($\upvarphi,1$) is satisfied. 
			By direct computation, it follows $$B_{k,k}:=\la\ffi_k,B\ffi_k\ra=\frac{3^3 L_{l(k)}^2 \sqrt{3} m(k)^2}{(-1+36 m(k)^2) \pi},\ \ \ \ \ \forall k\in\N^*.$$ For $(k,j),(m,n)\in I:=\{(k,j)\in(\N^*)^2:j\neq k\}$ so that $(k,j)\neq(m,n)$ and $\mu_k-\mu_j-\mu_m+\mu_n= 0$, we have $$L_{l(k)}=L_{l(j)}=L_{l(m)}=L_{l(n)}.$$ Indeed, the identity $L_{l(k)}\neq L_{l(j)}$ is never verified as it would imply $$m(k)^2=\frac{L_{l(k)}^2m(j)^2}{L_{l(j)}^2}+\frac{L_{l(k)}^2m(m)^2}{L_{l(m)}^2}-\frac{L_{l(k)}^2m(n)^2}{L_{l(n)}^2}\not\in\N^*.$$ 
			\begin{osss}\label{reciprocal} We notice that, for every $a,b,c,d\in\R$ different numbers, such that $a+b=c+d$, it holds $1/a+1/b\neq 1/c+1/d$. Indeed, we have $$1/a+1/b= (b+a)/(ab)=(d+c)/(ab)\neq (d+c)/(cd)=1/c+1/d,\ \ \ \ \ \text{ if }\ \ \ \ cd\neq ab.$$ Now, if $cd= ab$, then $a^2-c^2=d^2-b^2$ and $a+c=d+b$ since $a-c=d-b$, which is impossible as $2a\neq 2d$.\end{osss}
			
			\noindent
			In conclusion, $\mu_k-\mu_j-\mu_m+\mu_n= 0$ implies $k^2-j^2-m^2+n^2= 0$ and then $$k^{-2}-j^{-2}-m^{-2}+n^{-2}\neq 0.$$ Thus, $B_{k,k}-B_{j,j}-B_{m,m}+B_{n,n}\neq 0$ and Assumptions I($\upvarphi,1$) is valid.

			\smallskip
			
			\noindent{\bf Assumptions II.1 and conclusion.} Theorem $\ref{globalenergetic}$ leads to the statement since the point {\bf 2.\,\,}of Assumptions I($\upvarphi,1$) is satisfied thanks to Proposition $\ref{bor}$. Indeed, $B$ stabilizes $H^m$ for every $m>0$ and $H^2_\Gi(\upvarphi)$ since, for every $\psi\in H^2_\Gi(\upvarphi)$, $$(B\psi)^1(L_1)=(B\psi)^2(L_1)=(B\psi)^3(L_2)=(B\psi)^4(L_2)=0,$$ $$\dd_x(B\psi)^1(L_1)+\dd_x(B\psi)^2(L_1)+\dd_x(B\psi)^3(L_2)+\dd_x(B\psi)^4(L_2)=0. 
			$$ \qedhere
\end{proof}
\begin{osss}\label{generic}
As in \cite[Section\ 6]{mio4}, the techniques just developed are valid when $\Gi$ contains suitable sub-graphs 
denoted \virgolette{uniform chains}. A {uniform chain} is a sequence of edges of equal length $L$ connecting 
$M\in\N^*$ vertices $\{v_j\}_{j\leq M}$ such that $v_2,...,v_{M-1}\in V_i$. We also assume that either 
$v_1,v_{M}\in V_e$ are equipped with ($\Di$), $v_1=v_{M}\in V_i$, or $M=3$ and $v_1,v_3\in V_e$ are equipped with 
($\NN$). 
\begin{figure}[H]
\centering

\includegraphics[width=\textwidth-50pt]{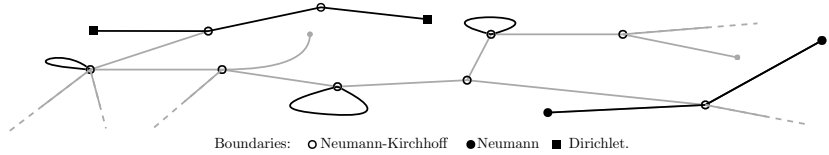}
\caption{Uniform chains contained in a generic graph.}\label{bas3}
\end{figure}

\noindent
Let $\Gi$ contain ${\widetilde N}\in\N^*$ {\it uniform chains} $\{\widetilde\Gi_j\}_{j\leq\widetilde N}$, composed by edges of lengths $\{L_j\}_{j\leq \widetilde N}\in\AL\LL(\widetilde N)$. Let $I_1\subseteq\{1,...,\widetilde N\}$ and $I_2\subseteq\{1,...,\widetilde N\}\setminus I_1$ be respectively the sets of indices $j$ such that the external vertices of $\widetilde\Gi_j$ are equipped with $(\NN)$ and $(\Di)$, while $I_3:=\{1,...,\widetilde N\}\setminus (I_1\cup I_2)$. 
If $\{L_j\}_{j\leq \widetilde N}\in\AL\LL(\widetilde N)$, then the energetic controllability can be guaranteed in 
{\small{	$$\Big\{\frac{(2k-1)^2\pi^2}{4L_j^2}\Big\}_{\underset{j\in I_1}{k,j\in\N^*}}\cup \Big\{\frac{k^2\pi^2}{L^2_j}\Big\}_{\underset{j\in I_2}{k,j\in\N^*}}\cup \Big\{\frac{(2k-1)^2\pi^2}{L^2_j}\Big\}_{\underset{j\in I_3}{k,j\in\N^*}}.$$}}
\end{osss}

\medskip
\noindent
{\bf Acknowledgments.} 
The second author has been financially supported by the ISDEEC project by ANR-16-CE40-0013.

\appendix\section{Analytic perturbation}\label{analitics}
We adapt the perturbation theory from \cite[Appendix\ B]{mio1} as done in \cite[Appendix\ C]{mio3}. Indeed, \cite{mio1} considers the (\ref{mainx1}) on $\Gi=(0,1)$ and $A$ is the Dirichlet Laplacian. As in \cite[Appendix\ B]{mio1}, we decompose $$u(t)=u_0+u_1(t),\ \ \ \ \ \ A+u(t)B=A+u_0B+u_1(t)B,\ \ \ \ \ u_0\in\R,\ u_1\in L^2((0,T),\R).$$ We consider $u_0B$ as a perturbative term of $A$.
Let us consider the (\ref{mainx1}) with $\Gi$ a quantum graph. Let $\upvarphi:=\{\ffi_k\}_{k\in\N^*}$ be an orthonormal system of $\Hi$ made by eigenfunctions of $A$ and let $\{\mu_{k}\}_{k\in\N^*}$ be the relative eigenvalues. Let $\{\ffi_j^{u_0}\}_{j\in\N^*}$ be an orthonormal system in $\Hi(\upvarphi):=\overline{span\{\varphi_k\ |\ k\in\N^*\}}^{\ L^2}$ made by eigenfunctions of $A+u_0B$ and $\{\mu_{k}^{u_0}\}_{k\in\N^*}$ be the relative eigenvalues. 

\begin{oss}
From (\ref{weakspectralgap}), we notice that there does not exist $\MM$ consecutive $k\in\N^*$ such that $|\mu_{k+1}  -   \mu_{k}| < \delta.$ This fact leads to a partition of $\N^*$ in subsets that we call $E_{m}$ with $m\in\N^*$. 
By definition, for every $m\in\N^*$, if $k,n\in E_m$, then $|\mu_{k}-\mu_n|<\delta (\MM-1),$
while if $k\in E_m$ and $n\not\in E_m$, then $|\mu_{k}-\mu_n|\geq \delta.$
This also defines an equivalence relation in $\N^*$ such that $k,n\in \N^*$ are equivalent if and only if there exists $m\in\N^*$ such that $k,n\in E_m.$ The sets $\{E_m\}_{m\in\N^*}$ are the corresponding equivalence classes and $i(m):=|E_m|\leq \MM-1$. 
\end{oss}

\smallskip
We denote as
$n:\N^*\rightarrow\N^*$ the application mapping $j\in\N^*$ in $n(j)\in\N^*$ such that $j\in E_{n(j)}$, while $s:\N^*\rightarrow \N^*$ is such that $\mu_{s(j)}=\inf\{\mu_k>\mu_j\ |\ k\notin E_{n(j)}\}$. Moreover, $p:\N^*\rightarrow \N^*$ is so that $\mu_{p(j)}=\sup\{k\in E_{n(j)}\}$.
Let $j\in\N^*$ and $P_j^{\bot}$ be the projector onto $\overline{span\{\ffi_m\ :\ m\not\in E_{n(j)}\}}^{\ L^2}.$ We define $\Pi:\Hi\rightarrow\Hi(\upvarphi)$ the orthogonal projector.

\begin{lemma}\label{bound}
	Let the hypotheses of Theorem $\ref{globalenergetic}$ be satisfied. There exists a neighborhood $U(0)$ of $u=0$ in $\R$ such that there exists $ c>0$ so that
	
	$$\iii((A+u_0B-\nu_k)\Pi)^{-1}\iii\leq c,\ \ \ \ \ \  \ \nu_k:=(\mu_{s(k)}-\mu_{p(k)})/2,\ \ \ \ \ \forall u_0\in U(0),\ \forall k\in\N^*.$$
	Moreover, for $u_0\in U(0)$, the operator $(A+u_0P_{k}^{\bot}B-\mu_k^{u_0})\Pi$ is invertible with bounded inverse from $H^2_\Gi(\upvarphi)\cap Ran(P_{k}^{\bot})$ to $Ran(P_{k}^{\bot})$ for every $ k\in\N^*$.
\end{lemma}
\begin{proof}The claim follows as \cite[Lemma\ B.2\ \&\ Lemma\ B.3]{mio1}.\end{proof} 
\begin{lemma}\label{chain}
	Let the hypotheses of Theorem \ref{globalenergetic} be satisfied. There exists a neighborhood $U(0)$ of $u=0$ in $\R$ such that, up to a countable subset $Q$ and for every $(k,j),(m,n)\in I:=\{(j,k)\in(\N^*)^2:j\neq k\},\ (k,j)\neq(m,n)$,
	$$\mu_k^{u_0}-\mu_j^{u_0}-\mu_m^{u_0}+\mu_n^{u_0}\neq 0,\ \ \ \ \ \ \ \ \la\ffi_k^{u_0},B\ffi_j^{u_0}\ra\neq 0,\ \ \ \ \ \ \  \ \forall u_0\in U(0)\setminus Q.$$
\end{lemma}
\begin{proof}
	For $k\in\N^*$, we decompose $\ffi_k^{u_0}=a_k\ffi_k+\sum_{j\in E_{n(k)}\setminus\{k\}}\beta_j^k\ffi_j +\eta_k,$ where $a_k\in\C$, $\{\beta_j^k\}_{j\in\N^*}\subset\C$ and $\eta_k$ is orthogonal to $\ffi_l$ for every $l\in E_{n(k)}$. Moreover, $\lim_{|u_0|\rightarrow 0}|a_k|=1$ and $\lim_{|u_0|\rightarrow 0}|\beta_j^k|=0$ for every $j,k\in\N^*$ and
	\begin{equation*}
	\begin{split}
	&\mu_k^{u_0}\ffi_k^{u_0}=
	(A+u_0B)(a_k\ffi_k+\sum_{j\in E_{n(k)}\setminus\{k\}}\beta_j^k\ffi_j+\eta_k)=Aa_k\ffi_k\\
	&+\sum_{j\in E_{n(k)}\setminus\{k\}}\beta_j^kA\ffi_j+A\eta_k+
	u_0Ba_k\ffi_k+u_0\sum_{j\in E_{n(k)}\setminus\{k\}}\beta_j^kB\ffi_j+u_0B\eta_k.\\
	\end{split}\end{equation*}
	Now, Lemma $\ref{bound}$ leads to the existence of $C_1>0$ such that, for every $k\in\N^*$,
	\begin{equation}\begin{split}\label{normeta}
	\eta_k=&-\big(\big(A+u_0P_{k}^{\bot}B-\mu_k^{u_0}
	\big)P_{k}^{\bot}\big)^{-1}u_0
	\Big(a_kP_{k}^{\bot}B\ffi_k+\sum_{j\in E_{n(k)}\setminus\{k\}}\beta_j^kP_{k}^{\bot}B\ffi_j\Big)
	\end{split}\end{equation}
	and $\|\eta_k\|\leq{C_1|u_0|}.$ 
	Let $B_{l,m}=\la\ffi_l,B\ffi_m$ for every $l,m\in\N^*$. We compute $\mu_k^{u_0}=\la\ffi_k^{u_0},(A+u_0B)\ffi_k^{u_0}\ra$ and
	\begin{equation*}\begin{split}
	&\mu_k^{u_0}=|a_k|^2\mu_k
	+\la\eta_k,(A+u_0B)\eta_k\ra+
	\sum_{j\in E_{n(k)}\setminus\{k\}}\mu_j|\beta_j^k|^2+u_0\sum_{j\in E_{n(k)}\setminus\{k\}}|\beta_j^k|^2B_{k,k}\\
	&+u_0\sum_{j,l\in E_{n(k)}\setminus\{k\}\ j\neq l}\overline{\beta_j^k}\beta_l^kB_{j,l}+u_0\sum_{j\in E_{n(k)}\setminus\{k\}}|\beta_j^k|^2(B_{j,j}-B_{k,k})+u_0|a_k|^2B_{k,k}\\
	&+2u_0 \Re\Big(\sum_{j\in E_{n(k)}\setminus\{k\}}\beta_j^k\la\eta_k,B\ffi_j\ra+\overline{a_k}
	\sum_{j\in E_{n(k)}\setminus\{k\}}\beta_j^kB_{k,j}+\overline{a_k}\la \ffi_k,B\eta_k\ra\Big).\\
	\end{split}\end{equation*}
	Thanks to $(\ref{normeta})$, it follows $\la\eta_k,(A+u_0B)\eta_k\ra=\mu_k^{u_0}\|\eta_k\|^2+O(u_0^2)$. Let 
	$$\widehat a_k:=\frac{|a_k|^2+\sum_{j\in E_{n(k)}\setminus\{k\}}|\beta_j^k|^2}{1-\|\eta_k\|^2},\ \ \ \ \ \ \widetilde a_k:=\frac{|a_k|^2+\sum_{j\in E_{n(k)}\setminus\{k\}}\mu_j/\mu_k|\beta_j^k|^2}{1-\|\eta_k\|^2}.$$ As $\|\eta_k\|\leq{C_1|u_0|}$, it follows $\lim_{|u_0|\rightarrow 0}|\widehat a_k|=1$ uniformly in $k$. Thanks to	$$\lim_{k\rightarrow +\infty}\inf_{j\in E_{n(k)}\setminus\{k\}} {\mu_j}{\mu_k}^{-1}=\lim_{k\rightarrow +\infty}\sup_{j\in E_{n(k)}\setminus\{k\}} {\mu_j}{\mu_k}^{-1}=1,$$
	we have $\lim_{|u_0|\rightarrow 0}|\widetilde a_k|=1$ uniformly in $k$.	Now, there exists $f_{k}$ such that
	\begin{equation}\label{perturbatedeigenvalues}\begin{split}
	\mu_k^{u_0}&=\widetilde a_k\mu_k+u_0\widehat a_kB_{k,k}+u_0f'_k+O(u_0^2)\\
	\end{split}\end{equation}
	where $\lim_{|u_0|\rightarrow 0}f_{k}=0$ uniformly in $k$. {When $\mu_k=0$, the identity $(\ref{perturbatedeigenvalues})$ is still valid}. For each $(k,j),(m,n)\in I$ such that $(k,j)\neq(m,n)$, there exists $f_{k,j,m,n}$ such that $\lim_{|u_0|\rightarrow 0}f_{k,j,m,n}=0$ uniformly in $k,j,m,n$ and
	\begin{equation*}
	\begin{split}
	&\mu_k^{u_0}-\mu_j^{u_0}-\mu_m^{u_0}+\mu_n^{u_0}=\widetilde a_k\mu_k-\widetilde a_j\mu_j-\widetilde a_m\mu_m+\widetilde a_n\mu_n+u_0f_{k,j,m,n}\\
	&+u_0(\widehat a_kB_{k,k}-\widehat a_jB_{j,j}-\widehat a_mB_{m,m}+\widehat a_nB_{n,n})=\widetilde a_k\mu_k-\widetilde a_j\mu_j\\
	&-\widetilde a_m\mu_m+\widetilde a_n\mu_n+u_0(\widehat a_kB_{k,k}-\widehat a_jB_{j,j}-\widehat a_mB_{m,m}+\widehat a_nB_{n,n})+O(u_0^2).
	\end{split}
	\end{equation*}
	Thanks to the third point of Assumptions I, there exists $U(0)$ a neighborhood of $u=0$ in $\R$ small enough such that, for each $u\in U(0)$, we have that every function $\mu_k^{u_0}-\mu_j^{u_0}-\mu_m^{u_0}+\mu_n^{u_0}$ is not constant and analytic.
	Now, $V_{(k,j,m,n)}=\{u\in D\big|\ \mu_k^{u}-\mu_j^{u}-\mu_m^{u}+\mu_n^{u}=0\}$ is a discrete subset of $D$ and
	$$V=\{u\in D\big|\ \exists ((k,j),(m,n))\in I^2: \mu_k^{u}-\mu_j^{u}-\mu_m^{u}+\mu_n^{u}=0\}$$
	is a countable subset of $D$, which achieves the proof of the first claim. The second relation is proved with the same technique. For $j,k\in\N^*$, the analytic function $u_0\rightarrow \la\ffi_j^{u_0},B\ffi_k^{u_0}\ra$ is not constantly zero since $\la\ffi_j,B\ffi_k\ra\neq 0$ and $W=\{u\in D\big|\ \exists (k,j)\in I: \la\ffi_j^{u_0},B\ffi_k^{u_0}\ra=0\}$ is a countable subset of $D$. \qedhere
\end{proof}

\begin{lemma}\label{equi}
	Let the hypotheses of Theorem $\ref{globalenergetic}$ be satisfied. Let $T>0$ and $s=d+2$ for $d$ introduced in Assumptions II. Let $c\in\R$ such that $0\not\in\sigma(A+u_0B+c,\Hi(\upvarphi))$ (the spectrum of $A+u_0B+c$ in the Hilbert space $\Hi(\upvarphi)$) and such that $A+u_0B+c$ is a positive operator. There exists a neighborhood $U(0)$ of $0$ in $\R$ such that,
	\begin{equation}\label{cazr}\begin{split}\forall u_0\in U(0),\ \  \ \ \ \ \Big\||A+u_0B+c|^\frac{s}{2}\psi\Big\|\ 
	\asymp\ \|\psi\|_{(s)},\ \ \ \ \ \ \forall\psi \in H^s_\Gi(\upvarphi).\end{split}\end{equation}\end{lemma}
\begin{proof}
	Let $D$ be the neighborhood provided by Lemma \ref{chain}. The proof follows the one of \cite[Lemma\ B.6]{mio1}. We suppose that $0\not\in\sigma(A+u_0B,\Hi(\upvarphi))$ and $A+u_0B$ is positive such that we can assume $c=0$. If $c\neq 0$, then the proof follows from the same arguments.

	\smallskip
	
	Thanks to Remark \ref{normequivalence}, we have $\|\cdot\|_{(s)}\asymp \||A|^\frac{s}{2}\cdot\|$ in $H^s_\Gi(\upvarphi)$. We prove the existence of $C_1,C_2,C_3>0$ such that, for every $\psi\in H^s_\Gi(\upvarphi)$,
	\begin{equation}\label{upperboundnormAuB}\begin{split}
	\|(A+u_0B)^\frac{s}{2}\psi\|\leq C_1\| A^\frac{s}{2}\psi\|+C_2\|\psi\|\leq C_3\| A^\frac{s}{2}\psi\|.\\
	\end{split}\end{equation}
	Let $s/2=k\in\N^*$. 
	The relation $(\ref{upperboundnormAuB})$ is proved by iterative argument. First, it is true for $k=1$ when $B\in L(H^2_\Gi(\upvarphi))$ as there exists $C>0$ such that $\| AB\psi\|\leq C\iii B\iii_{L(H^2_\Gi(\upvarphi))}\| A\psi\|$ for $\psi\in H^2_\Gi(\upvarphi)$. When $k=2$ if $B\in L(\Hi(\upvarphi))$ and $B\in L(H^{2 k_1}_\Gi(\upvarphi))$ for $1\leq k_1\leq 2$, then there exist $C_4,C_5>0$ such that, for $\psi\in H^4_\Gi(\upvarphi)$,
	\begin{equation*}\begin{split}&\|(A+u_0B)^2\psi\|\leq \| A^2\psi\|+|u_0|^2\|B^2\psi\|+|u_0|\|AB\psi\|+|u_0|\|BA\psi\|\\
	&\leq \| A^2\psi\|+|u_0|^2\iii B^2\iii_{L(\Hi(\upvarphi))} \|\psi\|+C_4|u_0|\iii B\iii_{L(H^{2 k_1}_\Gi(\upvarphi))}\|\psi\|_{(k_1)}+
	\\
	&|u_0|\iii B\iii_{L(\Hi(\upvarphi))}\|\psi\|_{(2)}
	\end{split}\end{equation*}
	and $\|(A+u_0B)^2\psi\|\leq C_5\| A^2\psi\|$. Second, we assume $(\ref{upperboundnormAuB})$ be valid for $k\in\N^*$ when $B\in L(H^{2 k_j}_\Gi(\upvarphi))$ for $k-j-1\leq k_j\leq k-j$ and for every $j\in\{0,..., k-1\}$. We prove $(\ref{upperboundnormAuB})$ for $k+1$ when $B\in L(H^{2 k_j}_\Gi(\upvarphi))$ for $k-j\leq k_j\leq k-j+1$ and for every $j\in\{0,..., k\}$. Now, there exists $C>0$ such that $\| A^{k}B\psi\|\leq C\iii B\iii_{L(H^{2 k_0}_\Gi(\upvarphi)))}\| A^{k_0}\psi\|$ for every $\psi\in H^{2 (k+1)}_\Gi(\upvarphi)$. Thus, as $\|(A+u_0B)^{k+1}\psi\|=\|(A+u_0B)^k(A+u_0B)\psi\|$, there exist $C_6, C_7>0$ such that, for every $\psi\in H^{2 (k+1)}_\Gi(\upvarphi)$,
	\begin{equation*}
	\begin{split}
	&\|(A+u_0B)^{k+1}\psi\|\leq C_6(\| A^{k+1}\psi\|+|u_0|\| A^{k}B\psi\|+\|A\psi\|+|u_0|\|B\psi\|)\leq C_{7}\| A^{k+1}\psi\|.\\
	\end{split}\end{equation*}
	As in the proof of \cite[Lemma\ B.6]{mio1}, the relation (\ref{upperboundnormAuB}) is valid for any $s\leq k$ when $B\in L(H^{2 k_0}_\Gi(\upvarphi))$ for $k-1\leq k_0\leq s$ and $B\in L(H^{2 k_j}_\Gi(\upvarphi))$ for $k-j-1\leq k_j\leq k-j$ and for every $j\in\{1,..., k-1\}$. The opposite inequality follows by decomposing $$A=(A+u_0B)-u_0B.$$
	
	\smallskip
	
	In our framework, Assumptions II ensure that the parameter $s$ is equal to $2+d$.  
	
	\noindent
	If the second point of Assumptions II is verified for $s\in[4,11/2)$, then $B$ preserves $H_{\NN\KK}^{d_1}(\upvarphi)$ and $H^{2}_{\Gi}(\upvarphi)$ for $d_1$ introduced in Assumptions II. Proposition \ref{bor} claims that $B:H^{d_1}_{\Gi}(\upvarphi)\rightarrow H^{d_1}_\Gi(\upvarphi)$ and the argument of \cite[Remark\ 2.1]{mio1} implies $B\in L(H^{d_1}_{\Gi}(\upvarphi))$ (also $B\in L(\Hi(\upvarphi))$ as $B:\Hi(\upvarphi)\longrightarrow\Hi(\upvarphi)$). Thus, the identity (\ref{cazr}) is valid because $B\in L(\Hi(\upvarphi))$, $B\in L(H^{2}_{\Gi}(\upvarphi))$ and $B\in L(H^{d_1}_{\Gi}(\upvarphi))$ with $d_1>s-2$.
	If the third point of Assumptions II is verified for $s\in[4,9/2)$, then $B\in L(\Hi(\upvarphi))$, $B\in L(H^{2}_{\Gi}(\upvarphi))$ and $B\in L(H^{d_1}_{\Gi}(\upvarphi))$ for $d_1\in[d,9,2)$. The claim follows thanks to Proposition \ref{bor} since $B$ stabilizes $H^{d_1}$ and $H^{2}_{\Gi}(\upvarphi)$ for $d_1$ introduced in Assumptions II.
	If $s<4$ instead, then the conditions $B\in L(\Hi(\upvarphi))$ and $B\in L(H^2_{\Gi}(\upvarphi))$ are sufficient to guarantee (\ref{cazr}).\qedhere
\end{proof}

\begin{osss}\label{nonna}
	The techniques developed in the proof of Lemma $\ref{equi}$ imply the following claim. Let the hypotheses of Theorem $\ref{globalenergetic}$ be satisfied and $0< s_1<d+2$ for $d$ introduced in Assumptions II. Let $c\in\R$ such that $0\not\in\sigma(A+u_0B+c,\Hi(\upvarphi))$ and such that $A+u_0B+c$ is a positive operator. We have There exists a neighborhood $U(0)\subset\R$ of $0$ so that, for any $u_0\in U(0)$, we have $$\||A+u_0B+c|^\frac{s_1}{2}\psi\|\ 
	\asymp\ \|\psi\|_{(s_1)},\ \ \ \  \ \forall\psi \in H^{s_1}_\Gi(\upvarphi).$$\end{osss}

\section{Global approximate controllability }\label{approximatecon}
Let us consider the notation introduced in Section $\ref{preli}$.

\begin{defi}
	The \eqref{mainx1} is said to be globally approximately controllable in $H_{\Gi}^{s}(\upvarphi)$ with $s>0$ when, for every $\psi\in H^{s}_{\Gi}(\upvarphi)$, $\widehat\G\in U(\Hi(\upvarphi))$ such that $\widehat\G\psi\in H^{s}_{\Gi}(\upvarphi)$ and $\epsilon>0$, there exist $T>0$ and $u\in L^2((0,T),\R)$ such that $\|\widehat\G\psi-\G^u_T\psi\|_{(s)}<\epsilon$.
\end{defi}
\begin{prop}\label{approx}
Let $(A,B)$ satisfy Assumptions I$(\upvarphi,\eta)$ and Assumptions II$(\upvarphi,\eta,\tilde d)$ for $\eta>0$ and $\widetilde d\geq 0$. The (\ref{mainx1}) is globally approximately controllable in $H^{s}_{\Gi}(\upvarphi)$ for $s=2+d$ with $d$ from Assumptions II$(\upvarphi,\eta,\tilde d)$ .
\end{prop}
\begin{proof}
In the point {\bf 1)\,\,}of the proof, we suppose that $(A,B)$ admits a non-degenerate chain of connectedness (see \cite[Definition\ 3]{nabile}). We treat the general case in the point {\bf 2)\,\,}.

\smallskip

\needspace{3\baselineskip}

\noindent
{\bf 1) (a) Preliminaries.} Let $\pi_m$ be the orthogonal projector $\pi_m:\Hi\rightarrow \Hi_m:=span\overline{\{\varphi_j\ :\ j\leq m\}}^{\ L^2}$ for every $m\in\N^*.$
Up to reordering of $\{\varphi_k\}_{k\in\N^*}$, the couples $(\pi_{m}A\pi_{m},\pi_{m}B\pi_{m})$ for $m\in\N^*$ admit non-degenerate chains of connectedness in $\Hi_{m}$. Let $\|\cdot\|_{BV(T)}=\|\cdot\|_{BV((0,T),\R)}$ and $\iii\cdot\iii_{(s)}:=\iii\cdot\iii_{L(H^s_{\Gi}(\upvarphi),H^s_{\Gi}(\upvarphi))}$ for $s>0.$ 
\begin{itemize}
	\item[]	 {\bf Claim.} $\forall\ \widehat\G\in U(\Hi(\upvarphi)),\ \forall \epsilon>0,\ \exists N_1\in\N^*,\ \widetilde\G_{N_1}\in U(\Hi(\upvarphi))\ :\ \pi_{N_1}\widetilde\G_{N_1}\pi_{N_1}\in SU(\Hi_{N_1}),$
	\begin{equation}\label{grammo}\|\widetilde\G_{N_1}\varphi_1-\widehat\Gamma\varphi_1\|<\epsilon.\end{equation}
\end{itemize}
Let $N_1\in\N^*$ and $\widetilde\varphi_1:=\|\pi_{N_1}\widehat\G\varphi_1\|^{-1}\pi_{N_1}\widehat\G\varphi_1$. We define $(\widetilde\varphi_j)_{2\leq j\leq N_1}$ such that $(\widetilde\varphi_j)_{j\leq N_1}$ is an orthonormal basis of $\Hi_{N_1}$. The operator $\widetilde\G_{N_1}$ is the unitary map such that $\widetilde\G_{N_1}\varphi_j=\widetilde\varphi_j$ for every $j\leq N_1.$	The provided definition implies $\lim_{N_1\rightarrow\infty}\|\widetilde\G_{N_1}\varphi_1-\widehat\Gamma\varphi_1\|=0$. Thus, for every $\epsilon>0$, there exists $N_1\in\N^*$ large enough satisfying the claim.

\smallskip
\needspace{3\baselineskip}

\noindent
{\bf 1) (b)  Finite dimensional controllability.} Let $T_{ad}$ be the set of $(j,k)\in\{1,...,N_1\}^2$ such that $B_{j,k}:=\la\varphi_j,B\varphi_k\ra\neq 0$ and $|\lambda_j-\lambda_k|=|\lambda_m-\lambda_l|$ with $m,l\in\N^*$ implies $\{j,k\}= \{m,l\}$ for $B_{m,l}=0$.
For every $(j,k)\in\{1,...,N_1\}^2$ and $\theta\in[0,2\pi)$, we define $E_{j,k}^{\theta}$ the $N_1\times N_1$ matrix with elements $(E_{j,k}^\theta)_{l,m}=0$, $(E_{j,k}^\theta)_{j,k}=e^{i\theta}$ and $(E_{j,k}^\theta)_{k,j}=-e^{-i\theta}$ for $(l,m)\in\{1,...,N_1\}^2\setminus\{(j,k),(k,j)\}.$ Let $E_{ad}=\big\{E_{j,k}^{\theta}\ :\ (j,k)\in T_{ad},\ \theta\in[0,2\pi)\big\}$ 
and $Lie(E_{ad})$. Fixed $v$ a piecewise constant control taking value in $E_{ad}$ and $\tau>0$, we introduce the control system on $SU(\Hi_{N_1})$
\begin{equation}\label{formulapprox3}\begin{split}\begin{cases}
\dot{x}(t)=x(t)v(t),\ \ \ \ \ \ t\in(0,\tau),\\
x(0)=Id_{SU(\Hi_{N_1})}.\\
\end{cases}\end{split}\end{equation}

\begin{itemize}
	\item[] {\bf Claim.} $(\ref{formulapprox3})$ is controllable, {\it i.e.} for $R\in SU(\Hi_{N_1})$, there exist $p\in\N^*$, $M_1,...,M_p\in E_{ad}$, $\alpha_1,...,\alpha_p\in\R^+$ such that $R=e^{\alpha_1 M_1}\circ...\circ e^{\alpha_p M_p}.$
\end{itemize}
For every $(j,k)\in\{1,...,N_1\}^2$, we define the $N_1\times N_1$ matrices $R_{j,k}$, $C_{j,k}$ and $D_{j}$ as follow. For $(l,m)\in\{1,...,N_1\}^2\setminus\{(j,k),(k,j)\},$ we have $(R_{j,k})_{l,m}=0$ and $(R_{j,k})_{j,k}=-(R_{j,k})_{k,j}=1,$ while $(C_{j,k})_{l,m}=0$ and $(C_{j,k})_{j,k}=(C_{j,k})_{k,j}=i.$ Moreover, for $(l,m)\in\{1,...,N_1\}^2\setminus\{(1,1),(j,j)\},$ $(D_{j})_{l,m}=0$ and $(D_{j})_{1,1}=-(D_{j})_{j,j}=i.$
We consider the basis of $su(\Hi_{N_1})$ $${\bf e}:=\{R_{j,k}\}_{j,k\leq N_1}\cup\{C_{j,k}\}_{j,k\leq N_1}\cup\{D_{j}\}_{j\leq N_1}.$$ Thanks to \cite[Theorem\ 6.1]{sac}, the controllability of (\ref{formulapprox3}) is equivalent to prove that $Lie(E_{ad})\supseteq su(\Hi_{N_1})$ for $su(\Hi_{N_1})$ the Lie algebra of $SU(\Hi_{N_1}).$ The claim si valid as it is possible to obtain the matrices  $R_{j,k}$, $C_{j,k}$ and $D_j$ for every $j,k\leq N_1$ by iterated Lie brackets of elements in $E_{ad}$.

	\smallskip

	\needspace{3\baselineskip}
	
	\noindent
	{\bf 1) (c) Finite dimensional estimates.} Let $\widehat\G\in U(\Hi(\upvarphi))$ and $ \widetilde\G_{N_1}\in U(\Hi(\upvarphi))$ be defined in {\bf 1) (a)}. Thanks to the previous claim and to the fact that $\pi_{N_1}\widetilde\G_{N_1}\pi_{N_1}\in SU(\Hi_{N_1})$, there exist $p\in\N^*$, $M_1,...,M_p\in E_{ad}$ and $\alpha_1,...,\alpha_p\in\R^+$ such that
	\begin{equation}\label{dexk}\pi_{N_1}\widetilde\G_{N_1}\pi_{N_1}=e^{\alpha_1 M_1}\circ...\circ e^{\alpha_p M_p}.\end{equation}
	
	\begin{itemize}\item[] {\bf Claim.} For every $l\leq p$ and $e^{\alpha_l M_l}$ from $(\ref{dexk})$, there exist $\{T_n^l\}_{l\in\N^*}\subset\R^+$ and $\{u_n^l\}_{n\in\N^*}$ such that $u_n^l\in L^2((0,T_n^l),\R)$ for every $n\in\N^*$ and
		\begin{equation}\label{sorde}\lim_{n\rightarrow\infty}\|\G_{T_n^l}^{u_n^l}\varphi_k-e^{\alpha_l M_l}\varphi_k\|=0,\ \ \ \  \ \ \forall k\leq N_1, \end{equation}
		\begin{equation}\label{sorde1}\begin{split}\sup_{n\in\N^*}\|u_n^l&
		\|_{BV(T_n^l)}<\infty,  \ \ \  \ \ \ \ \ \sup_{n\in\N^*}\|u_n^l\|_{L^\infty((0,T_n^l),\R)}<\infty,\\ 
		&\sup_{n\in\N^*} T_n^l\|u_n^l\|_{L^\infty((0,T_n^l),\R)}<\infty.\end{split}\end{equation}\end{itemize}
	We consider the results developed in \cite[Section\ 3.1\ \&\ Section\ 3.2]{chambrion2} by Chambrion and leading to \cite[Proposition\ 6]{chambrion2} since $(A,B)$ admits a non-degenerate chain of connectedness (\cite[Definition\ 3]{nabile}). Each $e^{\alpha_l M_l}$ is a rotation in a two dimensional space for every $l\in\{1,...,p\}$ and this work allows to explicit $\{T_n^l\}_{n\in\N^*}\subset\R^+$ and $\{u_n^l\}_{n\in\N^*}$ satisfying $(\ref{sorde1})$ such that $u_n^l\in L^2((0,T_n^l),\R)$ for every $n\in\N^*$ and \begin{equation}\label{diid}\lim_{n\rightarrow\infty}\|\pi_{N_1}\G_{T_n^l}^{u_n^l}\varphi_k-e^{\alpha_l M_l}\varphi_k\|=0,\ \ \ \ \ \ \forall k\leq N_1.\end{equation}
	As $e^{\alpha_l M_l}\in SU(\Hi_{N_1})$,  we have $\lim_{n\rightarrow\infty}\|\G_{T_n^l}^{u_n^l}\varphi_k-e^{\alpha_l M_l}\varphi_k\|=0$ for $k\leq N_1.$

	\smallskip

	\needspace{3\baselineskip}
	
	\noindent
	{\bf 1) (d) Infinite dimensional estimates.} 
	
	\begin{itemize}
		\item[] {\bf Claim.} Let $\widehat\G\in U(\Hi(\upvarphi))$. There exist $K_1,K_2,K_3>0$ such that for every $\epsilon>0$, there exist $T>0$ and $u\in L^2((0,T),\R)$ such that $\|\G_{T}^{u}\varphi_1-\widehat\G\varphi_1\|\leq\epsilon$ and
		\begin{align}\label{casino}\|u\|_{BV(T)}\leq K_1,  \ \  \ \ \ \ \ \|u\|_{L^\infty((0,T),\R)}\leq K_2,\ \  \ \ \ \ \ T\|u\|_{L^\infty((0,T),\R)}\leq K_3.\end{align}
	\end{itemize}
	
	Let {\bf 1) (c)} be valid with $p=2$. Although, the following result is valid for any $p\in\N^*$. There exists $2\leq l\leq N_1$ such that $e^{\alpha_1 M_1}\varphi_1=\varphi_l$. Thanks to $(\ref{sorde})$, there exists $n\in\N^*$ large enough such that,
	\begin{equation*}\begin{split}&\| \G_{T_n^2}^{u_n^2}\G_{T_n^1}^{u_n^1}\varphi_1-e^{\alpha_2 M_2} e^{\alpha_1 M_1}\varphi_1\|
\leq \iii \G_{T_n^2}^{u_n^2}\iii\|\G_{T_n^1}^{u_n^1}\varphi_1-e^{\alpha_1M_1}\varphi_1\|+\|\G_{T_n^2}^{u_n^2}\varphi_l-e^{\alpha_2M_2}\varphi_l\|\leq \epsilon.\end{split}\end{equation*}
The identity $(\ref{dexk})$ leads to the existence of $K_1,K_2,K_3>0$ such that for every $\epsilon>0$, there exist $T>0$ and $u\in L^2((0,T),\R)$ such that $\|\G_{T}^{u}\varphi_1-\widetilde\G_{N_1}\varphi_1\|<\epsilon$ and \begin{equation}\label{sorde2}\begin{split}\|u&\|_{BV(T)}\leq K_1,   \ \ \ \ \ \|u\|_{L^\infty((0,T),\R)}\leq K_2,\  \ \ \ \ T\|u\|_{L^\infty((0,T),\R)}\leq K_3.\end{split}\end{equation} The relation $(\ref{grammo})$ and the triangular inequality achieve the claim.

	\smallskip
	\noindent
	{\bf 1) (e) Global approximate controllability with respect to the $L^2$-norm.} Let $\psi\in \Hi(\upvarphi)$ and $\widehat \G\in U(\Hi(\upvarphi))$.

	\begin{itemize}
		\item[] {\bf Claim.} There exist $K_1,K_2,K_3>0$ such that for every $\epsilon>0$, there exist $T>0$ and $u\in L^2((0,T),\R)$ such that $\|\G_{T}^{u}\psi-\widehat\G\psi_k\|\leq\epsilon$ and
		\begin{align}\label{casino}\|u\|_{BV(T)}\leq K_1,  \ \  \ \ \ \ \ \|u\|_{L^\infty((0,T),\R)}\leq K_2,\ \  \ \ \ \ \ T\|u\|_{L^\infty((0,T),\R)}\leq K_3.\end{align}
	\end{itemize}

	We assume that $\|\psi \|=1$, but the same proof is also valid for the generic case.
	From the point {\bf 1) (d)}, there exist two controls respectively steering $\varphi_1$ close to $\psi$ and $\varphi_1$ close to $\widehat\G\psi$. Vice versa, thanks to the time reversibility, there exists a control steering $\psi$ close to $\varphi_1$. In other words, there exist $T_1, T_2>0$, $u_1\in L^2((0,T_1),\R)$ and $u_2\in L^2((0,T_2),\R)$ such that
	$$\|\G_{T_1}^{u_1}\psi-\varphi_1\|\leq{\epsilon},\ \ \ \  \ \ \ \ \ \|\G_{T_2}^{u_2}\varphi_1-\widehat\G\psi\|\leq{\epsilon}.$$
	The chosen controls $u_1$ and $u_2$ satisfy \eqref{casino}. The claim is proved as
	\begin{align*}\|\G_{T_2}^{u_2}\G_{T_1}^{u_1}\psi-\widehat\G\psi\|&\leq \|\G_{T_2}^{u_2}\G_{T_1}^{u_1}\psi-\G_{T_2}^{u_2}\varphi_1\|+\|\G_{T_2}^{u_2}\varphi_1-\widehat\G\psi\|
	\leq 2{\epsilon}.\end{align*}

	\noindent
	{\bf 1) (f) Global approximate controllability in higher regularity norm.} Let $\psi\in H^{s}_{\Gi}(\upvarphi)$ with $s\in[s_1,s_1+2)$ and $s_1\in\N^*$. Let $\widehat \G\in U(\Hi(\upvarphi))$ be such that $\widehat \G\psi\in H^{s}_{\Gi}(\upvarphi)$ and $B:H^{s_1}_{\Gi}(\upvarphi)\longrightarrow H^{s_1}_{\Gi}(\upvarphi)$.
	\begin{itemize}
		\item[] {\bf Claim.} There exist $T>0$ and $u\in L^2((0,T),\R)$ such that $\|\G_{T}^{u}\psi-\widehat\G\psi\|_{(s)}\leq\epsilon$.
	\end{itemize}
	
	We consider the propagation of regularity developed by Kato in \cite{kato1}. We notice that $i(A+u(t)B-ic)$ is maximal dissipative in $H^{s_1}_{\Gi}(\upvarphi)$ for suitable $c:=\|u\|_{L^\infty((0,T),\R)}\iii B\iii_{(s_1)}$. Let $\lambda>c$ and $\widehat H^{s_1+2}_{\Gi}(\upvarphi):=D(A^\frac{s_1}{2}(i\lambda-A))\cap\Hi(\upvarphi)\equiv  H^{s_1+2}_{\Gi}(\upvarphi)$. We know that $B:\widehat H^{s_1+2}_{\Gi}(\upvarphi)\subset  H^{s_1}_{\Gi}(\upvarphi)\rightarrow H^{s_1}_{\Gi}(\upvarphi)$ and the arguments of \cite[Remark\ 2.1]{mio1} imply that $B\in L(\widehat H^{s_1+2}_{\Gi}(\upvarphi),H^{s_1}_{\Gi}(\upvarphi))$. For $T>0$ and $u\in BV((0,T),\R)$, we have $$
	M:=\sup_{t\in [0,T]}\iii(i\lambda-A-u(t)B)^{-1}\iii_{L(H^{s_1}_{\Gi}(\upvarphi),\widehat H^{s_1+2}_{\Gi}(\upvarphi))}<+\infty.$$
We know $\|k+f(\cdot)\|_{BV((0,T),\R)}=\|f\|_{BV((0,T),\R)}$ for $f\in BV((0,T),\R)$ and $k\in\R$. Equivalently, $$N:=\iii i\lambda-A-u(\cdot)B\iii_{BV\big([0,T],L(\widehat H^{s_1+2}_{\Gi}(\upvarphi),H^{s_1}_{\Gi}(\upvarphi))\big)}=\|u\|_{BV(T)} \iii B\iii_{L(\widehat H^{s_1+2}_{\Gi}(\upvarphi),H^{s_1}_{\Gi}(\upvarphi))}<+\infty.$$ We call $C_1:=\iii A(A+u(T)B-i\lambda)^{-1}\iii_{(s_1)}<\infty$ and $U_t^{u}$ the propagator generated by $A+uB-ic$ such that $U_t^u\psi=e^{-ct}\G_t^u\psi$. Thanks to \cite[Section\ 3.10]{kato1}, for every $\psi\in H^{s_1+2}_{\Gi}(\upvarphi)$, it follows 
\begin{equation*}\begin{split}\|(A+u(T)B-i\lambda) U_t^u \psi\|_{(s_1)}\leq Me^{MN}\|(A-i\lambda) \psi\|_{(s_1)}\end{split}\end{equation*}
$$\Longrightarrow \ \ \ \|\G_{T}^{u} \psi\|_{(s_1+2)}\leq C_1 Me^{MN+cT}\|\psi\|_{(s_1+2)}.$$
For every $T>0$, $u\in BV((0,T),\R)$ and $\psi\in H^{s_1+2}_{\Gi}(\upvarphi)$, there exists $C=C(K)>0$ depending on $K=\big(\|u\|_{BV(T)},\|u\|_{L^\infty((0,T),\R)},T\|u\|_{L^\infty((0,T),\R)}\big)$ such that \begin{equation}\label{diid1}\|\G_{T}^{u} \psi\|_{(s_1+2)}\leq  C\|\psi\|_{(s_1+2)}.\end{equation}
Now, we notice that, for every $\psi\in H^{6}_{\Gi}(\upvarphi)$, from the Cauchy-Schwarz inequality, we have $\|A\psi\|^2
\leq\|\psi\|\|A^2\psi\|$ and there exists $C_2>0$ such that $\|A^2\psi\|^4\leq\|A\psi\|^2\|A^3\psi\|^2\leq C_2\|\psi\|\|A^3\psi\|^3$. By following the same idea, for every $\psi\in H^{s_1+2}_{\Gi}(\upvarphi)$, there exist $m_1,m_2\in\N^*$ and $C_3,C_4>0$ such that
\begin{equation}\label{diid2}\|A^\frac{s}{2}\psi\|^{m_1+m_2}\leq C_3\|\psi\|^{m_1}\|A^\frac{s_1+2}{2}\psi\|^{m_2} \ \ \ \ \ \Longrightarrow \ \ \ \ \|\psi\|_{(s)}^{m_1+m_2}\leq C_4\|\psi\|^{m_1}\|\psi\|_{(s_1+2)}^{m_2}.\end{equation}
	In conclusion, the point {\bf 1) (e)}, the relation $(\ref{diid1})$ and the relation $(\ref{diid2})$ ensure the claim.

\smallskip

\noindent
{\bf 1) (g)  Conclusion.} Let $d$ be defined in Assumptions II$(\upvarphi,\eta,\tilde d)$. 
If $d<2$, then $B:H^{2}_{\Gi}(\upvarphi)\rightarrow H^{2}_{\Gi}(\upvarphi)$ and the global approximate controllability is verified in $H^{d+2}_{\Gi}(\upvarphi)$ since $d+2<4.$ 
If $d\in [2,5/2)$, then $B:H^{d_1}\rightarrow H^{d_1}$ with $d_1\in(d,5/2)$ from Assumptions II$(\upvarphi,\eta,\tilde d)$. Now, $H^{d_1}_{\Gi}(\upvarphi)=H^{d_1}\cap H^2_{\Gi}(\upvarphi)$, thanks to Proposition \ref{bor}, and $B:H^{2}_{\Gi}(\upvarphi)\rightarrow H^{2}_{\Gi}(\upvarphi)$ implies $B:H^{d_1}_{\Gi}(\upvarphi)\rightarrow H^{d_1}_{\Gi}(\upvarphi)$. 
The global approximate controllability is verified in $H^{d+2}_{\Gi}(\upvarphi)$ since $d+2<d_1+2.$
If $d\in [5/2,7/2)$, then $B:H_{\NN\KK}^{d_1}(\upvarphi)\rightarrow H_{\NN\KK}^{d_1}(\upvarphi)$ for $d_1\in(d,7/2)$ and $H^{d_1}_{\Gi}(\upvarphi)=H^{d_1}_{\NN\KK}(\upvarphi)\cap H^2_{\Gi}(\upvarphi)$ from Proposition \ref{bor}. Now, $B:H^{2}_{\Gi}(\upvarphi)\rightarrow H^{2}_{\Gi}(\upvarphi)$ that implies $B:H^{d_1}_{\Gi}(\upvarphi)\rightarrow H^{d_1}_{\Gi}(\upvarphi)$. 
The global approximate controllability is verified in $H^{d+2}_{\Gi}(\upvarphi)$ since $d+2<d_1+2.$

\smallskip

\noindent
{\bf 2) Generalization.} Let $(A,B)$ do not admit a non-degenerate chain of connectedness and $$A+u(\cdot)B=(A+u_0B)+u_1(\cdot)B,\ \ \ \ \ \ \  \ \ \ \ \ \  \ u_0\in \R,\ \ \ \ u_1\in L^2((0,T),\R).$$ If $(A,B)$ satisfies Assumptions I and Assumptions II, then Lemma \ref{chain} and Lemma \ref{equi} (Appendix $\ref{analitics}$) are valid. We consider $u_0$ belonging to the neighborhoods provided by the two lemmas and we denote $(\ffi_k^{u_0})_{k\in\N^*}$ a Hilbert basis of $\Hi$ made by eigenfunctions of $A+u_0B$. The steps of the point {\bf 1)\,\,}can be repeated by considering the sequence $(\ffi_k^{u_0})_{k\in\N^*}$ instead of $(\ffi_k)_{k\in\N^*}$ and the spaces $D(|A+u_0B|^\frac{s_1}{2})\cap\Hi(\upvarphi)$ in substitution of $H^{s_1}_\Gi(\upvarphi)$ with $s_1>0$. The claim is equivalently proved thanks to Lemma \ref{equi}. \qedhere
\end{proof}

\begin{osss}\label{approxT}
	As Proposition $\ref{approx}$, the (\ref{mainT}) is globally approximately controllable in $H_{\Ti}^{3}(\upvarphi)$ (defined in $(\ref{HSTspaces})$). In other words, for every $\psi\in H^{3}_{\Ti}(\upvarphi)$, $\widehat\G\in U(\Hi(\upvarphi))$ such that $\widehat\G\psi\in H^{3}_{\Ti}(\upvarphi)$ and $\epsilon>0$, we have
	$$\exists T>0,\ u\in L^2((0,T),\R)\ \ \ :\ \ \ \|\widehat\G\psi_k-\G^u_T\psi_k\|_{(3)}<\epsilon.$$
	Indeed, for every $(j,k),(l,m)\in I:=\{(j,k)\in(\N^*)^2:j\neq k\}$ so that $(j,k)\neq(l,m)$ and such that $$\mu_j-\mu_k-\mu_l+\mu_m=\frac{\pi^2}{L^2}(j^2-k^2-l^2+m^2)=0,$$
	there exists $C>0$ so that, thanks to Remark $\ref{reciprocal},$ we have
	$$\la\ffi_j,B\ffi_j\ra-\la\ffi_k,B\ffi_k\ra-\la\ffi_l,B\ffi_l\ra+\la\ffi_m,B\ffi_m\ra=C(j^{-2}-k^{-2}-l^{-2}+m^{-2})\neq 0.$$
	In conclusion, the statement of Lemma $\ref{chain}$ is valid when $|u_0|$ is small enough. Thus, $(A+u_0B,B)$ admits a non-degenerate chain of connectedness. The arguments adopted in the proof of Proposition $\ref{approx}$ lead to the claim.
\end{osss}

\end{document}